\documentclass{amsart}
\usepackage{amsmath,amssymb,amsthm,amsfonts,amscd,mathrsfs}
\usepackage[all]{xy}
\usepackage{chngcntr}
\usepackage{tikz-cd}
 \usepackage{animate}
 \usepackage{graphicx} 
 \usepackage{pb-diagram}
 \usepackage[]{todonotes}

\usepackage{pgf,tikz}
\usepackage{mathrsfs}
\usetikzlibrary{arrows}
\usetikzlibrary{decorations.markings}

\title[Free and based path groupoids]{Free and based path groupoids}

\author{Andr\'es \'Angel}

\address{Departamento de Matem\'aticas, Universidad de los Andes, Carrera 1 N. 18A - 10, Bogot\'a, Colombia}
\email{ja.angel908@uniandes.edu.co}
\urladdr{}
\address{
	 Departamento de Matem\'aticas y Estad\'istica\\
       Universidad del Norte\\ 
       Km. 5 via Puerto Colombia, Barranquilla, Colombia}    
       \email{ajangel@uninorte.edu.co}

\author{Hellen Colman}

\address{Department of Mathematics, Wright College,
 4300 N.\ Narragansett Avenue, Chicago, IL 60634 USA}
\email{hcolman@ccc.edu}
\urladdr{https://www.hellencolman.com/math}

\theoremstyle{definition}
\newtheorem{defn}{Definition}[section]

\newtheorem{exam}[defn]{Example}
\newtheorem{remark}[defn]{Remark}

\theoremstyle{plain}
\newtheorem{thm}[defn]{Theorem}
\newtheorem{prop}[defn]{Proposition}

\newcommand{\bdm}{\begin{displaymath}}
\newcommand{\edm}{\end{displaymath}}

\newcommand{\X}{{\mathcal X}}
\newcommand{\Gd}{{\mathcal G}}

\newcommand{\Grd}{\mathsf{Gpd}}
\newcommand{\Top}{\mathsf{Top}}
\newcommand{\Grp}{\mathsf{Grp}}
\newcommand{\TG}{\mathsf{TopG}}

\newcommand{\MTG}{\mathsf{MTopG}}
\newcommand{\TrG}{\mathsf{TrG}}

    \newcommand{\MTrG}{\mathsf{MTrG}}

    \newcommand{\B}{\mathbf{B}}
\newcommand{\cK}{{\mathcal K}}
\newcommand{\simT}{~ {\sim}_{T} ~}
\newcommand{\e}{\epsilon}
\newcommand{\cL}{\mathcal L}
\newcommand{\Int}{\mathcal I}
\newcommand{\cH}{\mathcal H}
\newcommand{\cP}{{\mathcal P}}
\newcommand{\cJ}{{\mathcal J}}
\newcommand{\Y}{{\mathcal Y}}

\newcommand{\inv}{^{-1}}

\newcommand{\id}{\operatorname{id}}

\newcommand{\ev}{\operatorname{ev}}

\newcommand{\mmm}{\mbox{\rm Map}}

\newcommand{\colim}{\operatornamewithlimits{colim}}

\begin{document}

\begin{abstract} We give an explicit description of the free path and loop groupoids in the Morita bicategory of translation topological groupoids. We prove that the free path groupoid of a discrete group acting properly on a topological space $X$ is a translation groupoid given by the same group acting on the topological path space $X^I$.  We give a detailed description of based path and loop groupoids and show that both are equivalent to topological spaces. We also establish the notion of homotopy and fibration in this context.
\end{abstract}

\maketitle

\section{Introduction}
Our aim is to give an explicit description of the path object in the bicategory of translation topological groupoids. Our main application will be in the setting of orbifolds as groupoids.

We adopt the model developed by Moerdijk and Pronk \cite{Pronk1997} to describe orbifolds in terms of groupoids.  Essentially an orbifold is a Morita equivalence class of groupoids of a certain type, that we will call orbifold groupoids.

In this spirit, the right notion of morphism between orbifold groupoids is that of a generalized map. These generalized maps arise as morphisms in the bicategory of topological groupoids, functors and natural transformations when inverting the essential equivalences \cite{Pronk1996}. 

All orbifolds can be represented by a groupoid given by a certain type of action of a group $G$ on a topological space $X$. This representation $G\ltimes X$ is called translation groupoid. In particular we will be interested in {\em developable} orbifolds defined by a translation groupoid given by a discrete group acting properly on a space.

For these orbifolds, we use their groupoid characterization to obtain a description of the generalized maps from the interval to the orbifold  as a translation groupoid. We prove that the free path groupoid of the translation groupoid $G\ltimes X$ is the translation groupoid $G\ltimes X^I$. In fact we describe three different approaches resulting in three characterizations of the path groupoid: as a colimit of $G$-paths, as a groupoid of multiple $G$-paths and as a translation groupoid $G\ltimes X^I$. We prove that the three groupoids are equivalent. 

We show that this construction of the path groupoid is functorial and invariant under Morita equivalence.

The pullback along the diagonal of this model gives us as a particular case, the free loop groupoid which coincides with the descriptions given by Lupercio and Uribe \cite{Lupercio}, Adem, Leida and  Ruan \cite{Adem} and Noohi  \cite{Noohi1} in various contexts.

Moreover, we use this model to calculate the based groupoid of paths between two points. We prove that this groupoid is actually equivalent to a topological space.

Using our description of the path groupoid, we provide an explicit characterization for a homotopy between two generalized maps, as well as a definition for orbifold fibrations. We prove that the evaluation map  is both a groupoid homotopy and a groupoid fibration.

We organize the paper in the following manner.  In Section 2 we present some basic definitions and constructions for topological groupoids. We define translation groupoids and introduce the bicategory of translation groupoids resulting from inverting the essential equivalences. Section 3 introduces the model for orbifolds as groupoids that gives the setting for the construction of the path groupoid in the next section. Section 4 is devoted to the construction of the free path groupoid. We give here an explicit equivalence between all models for the path groupoid. We prove that this construction is functorial and invariant under Morita equivalence.
Section 5 provides a detailed description of the based path and loop groupoids and describe some examples. Section 6 concerns the characterization of the homotopy between generalized maps. In section 7 we provide a definition of groupoid fibration and prove that the evaluation morphism is a groupoid fibration.

\section{Context}
\subsection{Topological groupoids}

A  {\it topological groupoid} $\Gd$ is a groupoid object in the category $\Top$ of topological spaces and continuous maps.
Our notation for groupoids is that $G_0$ is the space of objects and $G_1$ is the space of arrows, with source and target maps
$s,t:G_1\to G_0$, multiplication $m:G_1 \times_{G_0} G_1\to G_1$, inversion $i:G_1\to G_1$, and object inclusion
$u:G_0\hookrightarrow G_1$.

The set of arrows from $x$ to $y$ is denoted $G(x,y)=\{ g\in G_1 | s(g)=x \mbox{ and } t(g)=y\}$. The set of arrows from $x$ to itself, $G(x,x)$, is a group called the {\it isotropy} group of $\Gd$ at $x$ and denoted $G_{x}$.

A {\it strict morphism} $\phi: \cK \to \Gd$ of groupoids is a functor given by two continuous  maps $\phi: K_1 \to G_1$ and $\phi: K_0 \to G_0$ that together commute with all the structure maps of the groupoids $\cK$ and $\Gd$.

A {\it natural transformation} $T : \phi \Rightarrow \psi$ between two morphisms $\phi, \psi: \cK \to \Gd$ is a continuous map $T: K_{0} \to G_1$ with $T(x):\phi(x)\to\psi(x)$ such that for any arrow $h:x \rightarrow y$ in $K_1$,  the identity $\psi(h)T(x)=T(y)\phi(h)$ holds. Since we are in a topological groupoid and inversion is continuous, we also have a natural transformation $T^{-1} : \psi  \Rightarrow \phi$ and write $\phi\simT\psi$.

Topological groupoids, strict morphisms and natural transformations form a 2-category that we denote $\TG$.

A strict morphism $\e: \cK \to \Gd$ of topological groupoids is an {\it essential equivalence}  if
\begin{itemize}
\item[(i)] $\e$ is essentially surjective in the sense that \[s\pi_{1}:G_1\times^t_{G_0}K_0\rightarrow G_0\] is an open surjection
where $G_1\times^t_{G_0}K_0$ is the pullback along the target $t: G_1\to G_0$;
\item[(ii)] $\e$ is fully faithful in the sense that $K_1$ is  the following pullback of topological spaces:
\[\xymatrix{
K_1 \ar[r]^{\e} \ar[d]_{(s,t)}& G_1 \ar[d]^{(s,t)} \\
K_0\times K_0 \ar[r]^{\e \times \e} & G_0\times G_0}\]
\end{itemize}

Note that if there exists a functor $\delta: \Gd \to \cK$ with natural transformations $\eta : \id_{\Gd} \Rightarrow  \e\circ\delta  $ and $\nu : \delta\circ\e \Rightarrow \id_{\cK}$ in $\TG$ the functor $\e$ is essentially surjective, indeed, $s \pi_1$ has a section defined by $(\eta_x,\delta(x)) : G_0 \rightarrow G_1\times^t_{G_0}K_0$ which implies that it is open and surjective. $\e$ is fully faithful because the map $K_1 \rightarrow \cK_0\times \cK_0 \times_{\Gd_0\times \Gd_0}  \Gd_1$ has an inverse defined by $(x,y,h) \rightarrow \nu_y \circ \delta (h) \circ \nu_x^{-1}$.

An essential equivalence $\e: \cK \to \Gd$ does not generally have an inverse functor $\delta: \Gd \to \cK$ such that $\e\circ\delta \simT \id_{\Gd}$ and $\delta\circ\e \sim_{T'} \id_{\cK}$ in $\TG$. The functor $\delta$ exists by the axiom of choice but in general it is not {\it continuous}.

\begin{defn}
Let $\psi:\cK\to \Gd$ and $\phi:\cL\to \Gd$ be strict morphisms. The \textit{groupoid pullback} $\cP=\cK\times_{\Gd}\cL$
is the topological groupoid whose space of objects is $P_0=K_0\times^t_{G_0}G_1\times^s_{G_0} L_0$ and space of arrows is 
$P_1=K_1\times^t_{G_0}G_1\times^s_{G_0} L_1$. Source and target maps are given by
$s(k,g,l)=(s(k), \psi(k)^{-1} g\phi (l), s(l)) $ and $t(k,g,l)=(t(k), g, t(l) )$.
 There is a square of morphisms and a natural transformation $T$ that makes the following diagram commutative

$$\xymatrix{\cK\times_{\Gd} \cL\ar[r]^-{\pi_{1}}\ar[d]_{\pi_{2}}="0"&\cK\ar[d]^{\psi}="1"\\ \cL\ar[r]_{\phi}&\Gd
\ar@{}"0" ;"1"_{\sim_{T}}
}$$
and is universal with this property.

\end{defn}

\begin{defn}
The groupoids $\cK$ and $\Gd$ are {\it Morita equivalent} if there exists a groupoid $\cL$ and a span 
\[\cK\overset{\sigma}{\gets}\cL\overset{\e}{\to}\Gd\] where $\e$ and $\sigma$ are essential equivalences. We write $\Gd\sim_M \cK$.
\end{defn}
The proof that a Morita equivalence is an equivalence relation is based in the groupoid pullback defined above.

A  {\it generalized map} $(\e,\phi)$ from $\cK$ to $\Gd$ is a span
$\cK\overset{\e}{\gets}\cJ\overset{\phi}{\to}\Gd$
such that $\e$ is an essential equivalence.
Two generalized maps $\cK\overset{\e}{\gets}\cJ\overset{\phi}{\to}\Gd$ and $\cK\overset{\e'}{\gets}\cJ'\overset{\phi'}{\to}\Gd$ are ${\it equivalent}$ if there exists a diagram
$$\xymatrix{ &
{\cJ}\ar[dr]^{\phi}="0" \ar[dl]_{\e }="2"&\\
{\cK}&{\cL} \ar[u]_{u} \ar[d]^{v}&{\Gd}\\
&{\cJ'}\ar[ul]^{\e'}="3" \ar[ur]_{\phi'}="1"&
\ar@{}"0";"1"|(.4){\,}="7"
\ar@{}"0";"1"|(.6){\,}="8"
\ar@{}"7" ;"8"_{\sim_{T'}}
\ar@{}"2";"3"|(.4){\,}="5"
\ar@{}"2";"3"|(.6){\,}="6"
\ar@{}"5" ;"6"^{\simT}
}
$$
which is commutative up to natural transformations and where $\cL$ is a topological groupoid, and $u$ and $v$ are essential  equivalences.


\subsection{The Morita bicategory of topological groupoids $\MTG$}

Consider the class of arrows $E$ given by the essential equivalences in the 2-category  $\TG$. It was proven by Pronk in \cite{Pronk1996, Pronk:2010} that $E$ satisfies the conditions to admit a bicalculus of fractions. The bicategory of fractions ${\TG}{(E\inv})$ obtained by formally inverting the essential equivalences is what we call the {\it Morita bicategory of topological groupoids} and we denote $\MTG$.

The explicit description of the bicategory $\MTG$ is as follows:

\begin{itemize}
\item Objects are topological groupoids $\Gd$.
\item A 1-morphism from $\cK$ to $\Gd$ is a  {\it generalized map}
\[\cK\overset{\e}{\gets}\cJ\overset{\phi}{\to}\Gd\]
such that $\e$ is an essential equivalence.
\item A 2-morphism from $\cK\overset{\e}{\gets}\cJ\overset{\phi}{\to}\Gd$ to $\cK\overset{\e'}{\gets}\cJ'\overset{\phi'}{\to}\Gd$ is given by a class of diagrams:
$$
\xymatrix{ &
{\cJ}\ar[dr]^{\phi}="0" \ar[dl]_{\e }="2"&\\
{\cK}&{\cL} \ar[u]_{u} \ar[d]^{v}&{\Gd}\\
&{\cJ'}\ar[ul]^{\e'}="3" \ar[ur]_{\phi'}="1"&
\ar@{}"0";"1"|(.4){\,}="7"
\ar@{}"0";"1"|(.6){\,}="8"
\ar@{}"7" ;"8"_{\sim_{T'}}
\ar@{}"2";"3"|(.4){\,}="5"
\ar@{}"2";"3"|(.6){\,}="6"
\ar@{}"5" ;"6"^{\simT}
}
$$
where $\cL$ is a topological groupoid, and $u$ and $v$ are essential  equivalences.
\end{itemize}

The horizontal composition of generalized maps $\cK\overset{\e}{\gets}\cJ\overset{\phi}{\to}\Gd$ and $\Gd\overset{\zeta}{\gets}\cJ'\overset{\psi}{\to}\cL$ is given by the diagram

\[
\xymatrix@!=1pc{&&\cJ'\times_{\Gd}\cJ\ar[dr]\ar[dl]&&\\
&\cJ\ar[dl]_{\e}\ar[dr]^{\phi}&&\cJ'\ar[dl]_{\zeta}\ar[dr]^{\psi}&\\
\cK &&\Gd&&\cL
}
\]
where $\cJ'\times_{\Gd}\cJ$ is the weak pullback of groupoids. Note that this composition is associative only up to a $2$-morphism.

\subsection{Translation groupoids}
Let $G$ be a topological group with a continuous  left action on a topological space $X$.
Then the {\em translation groupoid} $G\ltimes X$ is defined by:
\begin{itemize}   
\item The space of objects is  $X$ itself, and the space of arrows is the cartesian product  $G\times X$.    
\item The source $s: G\times X\to X$ is the second
projection, and the target $t:G\times X\to X$ is given by  the action. Then $(g, x) $ is an arrow $x \to gx$.  
\item The other structure maps are defined by the unit  $u(x)=(e,x)$, where $e$ is the identity element in $G$,
and $(h,gx)\circ(g,x)=(h\star g,x)$ where $\star$ is the group multiplication.
\end{itemize}

\begin{exam}
These examples will appear later on in our applications.
\begin{enumerate}
\item
{\it Unit groupoid.} 
Consider the groupoid $e\ltimes X$ given by the action of the trivial group $e$ on the topological space $X$. This is a topological groupoid whose arrows are all units. In this way, any topological space can be considered as a groupoid.
\item
{\it Multiplication groupoid.} Let $H$ be a subgroup of a topological group $G$. Consider the translation groupoid $H\ltimes G$ where $H$ acts by multiplication on $G$.
\item
{\it Conjugation groupoid.} Let $H$ be a subgroup of a topological group $G$. Consider the translation groupoid $H\ltimes G$ where $H$ acts by conjugation on $G$.
\item
{\it Point groupoid.} Let $G$ be a topological group. Let $\bullet$ be a point. Consider the groupoid $G\ltimes \bullet$ where $G$ acts trivially on the point. This is a topological groupoid with exactly one object $\bullet$ and $G$ is the space of arrows in which the maps $s$ and $t$ coincide. We call $G\ltimes \bullet$ the point groupoid associated to G. In this way any group can be considered as a groupoid.
\end{enumerate}
We will denote ${\bf{1}}$ the {\em trivial groupoid} with one object and one arrow, that is ${\bf{1}}=e\ltimes \bullet$ the unit groupoid over a point or a point groupoid associated to the trivial group. \end{exam}
An {\em equivariant map} $G\ltimes X\rightarrow K\ltimes Y$ between translation groupoids consists of a pair $\varphi\ltimes f$, 
where $\varphi\colon G\rightarrow K$
is a group homomorphism and $f\colon X\rightarrow Y$ satisfies $f(gx)=\varphi(g) f(x)$
for $g\in G$ and $x\in X$. 

Translation groupoids, equivariant maps and natural transformations form a 2-category that we denote $\TrG$.

\subsection{The Morita bicategory of translation groupoids $\MTrG$}
We construct now a sub  bicategory $\MTrG$ of the Morita bicategory of topological groupoids $\MTG$ where the objects are strictly the translation groupoids and the maps are equivariant ones.

\begin{prop}    \cite{Pronk:2010} Let $\psi: G\ltimes X\to L\ltimes Z$ and $\phi: H\ltimes Y\to L\ltimes Z$ be equivariant maps.
The fibre product $\cK$ 
$$\xymatrix{\cK\ar[r]^-{\pi_{1}\ltimes f}\ar[d]_{\pi_{2}\ltimes g}&G\ltimes X\ar[d]^{\psi}\\ H\ltimes Y\ar[r]^{\phi}&L\ltimes Z}$$
is again a translation groupoid. Moreover, its structure group is $G \times H$, $\cK=(G \times H) \ltimes P$ and the first components of the equivariant maps $\pi_{1}\ltimes f$ and $\pi_{2}\ltimes g$ are the group projections $\pi_{1}: G\times H\to G$ and $\pi_{2}: G\times H\to H$.
\end{prop}

An {\em equivariant essential equivalence} is an equivariant map $\xi\ltimes \e$ which is an essential equivalence.

Consider the bicategory whose 

\begin{itemize}
\item Objects are translation groupoids $G\ltimes X$.
\item 1-morphisms from $G\ltimes X$ to $K\ltimes Y$ are  {\it equivariant generalized maps}
\[G\ltimes X\overset{\xi\ltimes \e}{\gets}L\ltimes Z\overset{\varphi\ltimes f}{\to}K\ltimes Y\]
such that  $\xi\ltimes\e$ is an equivariant essential equivalence.
\item A 2-morphism $\Rightarrow$ from the equivariant generalized map $G\ltimes X\overset{\xi\ltimes\e}{\gets}L\ltimes Z\overset{\varphi\ltimes f}{\to}K\ltimes Y$ to $G\ltimes X\overset{\xi'\ltimes\e'}{\gets}L'\ltimes Z'\overset{\varphi'\ltimes f'}{\to}K\ltimes Y$ is given by a class of diagrams:
$$
\xymatrix{ &
{L\ltimes Z}\ar[dr]^{\varphi\ltimes f}="0" \ar[dl]_{\xi\ltimes\e }="2"&\\
{G\ltimes X}&{R\ltimes U} \ar[u]_{u} \ar[d]^{v}&{K\ltimes Y}\\
&{L'\ltimes Z'}\ar[ul]^{\xi'\ltimes\e'}="3" \ar[ur]_{\varphi'\ltimes f'}="1"&
\ar@{}"0";"1"|(.4){\,}="7"
\ar@{}"0";"1"|(.6){\,}="8"
\ar@{}"7" ;"8"_{\sim_{T'}}
\ar@{}"2";"3"|(.4){\,}="5"
\ar@{}"2";"3"|(.6){\,}="6"
\ar@{}"5" ;"6"^{\simT}
}
$$

where $R\ltimes U$ is a translation groupoid, and $u$ and $v$ are equivariant essential equivalences. \end{itemize}

Translation groupoids, equivariant generalized maps and diagrams as above form the {\em Morita bicategory of translation groupoids} that we denote $\MTrG$. 


\section{Orbifolds as groupoids}
We recall now the description of orbifolds as groupoids due to Moerdijk and Pronk \cite{Pronk1996, Pronk1997}. Orbifolds were first introduced by Satake \cite{S} as a generalization of a manifold defined in terms of local quotients. The groupoid approach provides a global language to reformulate the notion of orbifold.

 A groupoid $\Gd$ is {\it proper} if $(s,t):G_1\to G_0\times G_0$ is a proper map and it is a {\it foliation} groupoid if each isotropy group is discrete. 

\begin{defn}
An {\it orbifold} groupoid is a proper foliation groupoid.
\end{defn}

Given an orbifold groupoid $\Gd$, its orbit space $|\Gd|$ is a locally compact Hausdorff space. Given an arbitrary locally compact Hausdorff space $X$ we can equip it  with an orbifold structure as follows:

\begin{defn} An {\it orbifold structure} on a locally compact Hausdorff space $X$ is given by an orbifold groupoid $\Gd$ and a homeomorphism $h:|\Gd|\to X$.
\end{defn}

If $\e:\cH\to \Gd$ is an essential equivalence and $|\e|:|\cH|\to |\Gd|$ is the induced homeomorphism between orbit spaces, we say that the composition $h\circ|\e|:|\cH|\to X$ defines an {\it equivalent} orbifold structure. 

\begin{defn}  An {\it orbifold}  $\X$ 
is a space $X$ equipped with an equivalence class of orbifold 
structures. A specific such structure, given by 
$\Gd$ and  $h : |\Gd | \to X $ is
a {\it presentation} of the orbifold 
$\X$.
\end{defn}

If two groupoids are Morita equivalent, then they define the same orbifold. Therefore any structure or invariant for orbifolds, if defined through groupoids, should be invariant under Morita equivalence. 

\begin{defn}  An {\it orbifold map} $f\colon \Y\to \X$ is given by an equivalence class of generalized maps $(\e,\phi)$ from $\cK$ to $\Gd$ between presentations of the orbifolds such that the diagram commutes:

\[\xymatrix{
|\cK| \ar[r]^{|\phi||\e|^{-1}}\ar[d]& |\Gd| \ar[d]^{} \\ 
Y \ar[r]& X}\]
A specific such generalized map $(\e,\phi)$ is called a {\it presentation} of the orbifold map $f$.
\end{defn}

We can obtain an orbifold by consider the action of a compact group $G$ acting on a space $X$ with finite stabilizers. All orbifolds can be described in this way \cite{pardon}.

The orbifold $\X$ is {\em developable} is it is presented by a groupoid Morita equivalent to a translation groupoid $G\ltimes X$ with $G$ a discrete group acting properly on $X$.


\section{Path Groupoid}\label{path}
From now on, we will focus on developable orbifolds and $G$ will be a discrete group acting properly on $X$. In this context, we will show that in the bicategory of topological groupoids any path in $G\ltimes X$
$$I\overset{}{\gets}\Int \overset{}{\to}G\ltimes X$$
is equivalent to a strict map 
$$I \overset{}{\to}G\ltimes X$$
where $I$ is the unit groupoid $e\ltimes I$, $I=[0,1]$ and $\Int$ is any topological groupoid.

\subsection{Generalized paths}
A {\em path} in the groupoid $G\ltimes X$ in the Morita bicategory of topological groupoids is a generalized map $(\delta, \beta)$ from the unit groupoid $I$ to $G\ltimes X$. That is, a span 
$$I\overset{\delta}{\gets}\Int\overset{\beta}{\to}G\ltimes X.$$

Since $I\overset{\delta}{\gets}\Int$ is an essential equivalence, we can use groupoid atlases 
\cite{Pronk2017, thesis} to see that the equivalence class $[I\overset{\delta}{\gets}\Int\overset{\beta}{\to}G\ltimes X]$
has a representative of the form:

$$I\overset{\e}{\gets}I_{S_n}\overset{\alpha}{\to}G\ltimes X$$
where $I_{S_n}$ is the groupoid associated to a subdivision $$S_n=\{0=r_0\le r_1<\cdots<r_{n-1}\le r_n=1\}$$ of the interval $I=[0,1]$ as explained below.

The space of objects of the groupoid $I_{S_n}$  is the disjoint union
$$\bigsqcup_{i=1}^{n} I_i$$
where $I_i$ is a small open neighborhood of $[r_{i-1}, r_{i}]$ and $(r,i)$ denotes an element $r$ in the connected component  $I_i$.

The space of arrows of $I_{S_n}$ is given by the disjoint union $$\bigsqcup_{i=1}^{n}  I_i\bigsqcup_{i=1}^{n-1} ( \tilde{I}_i \sqcup \tilde{I}_i)$$

where $\displaystyle\bigsqcup_{i=1}^{n}  I_i$ is the set of unit arrows,
$ \tilde{I}_i= I_i\cap I_{i+1}$ and another copy $\tilde{I}_i$ was added for inverse arrows.
For each point $r_i$ in the subdivision ${S_n}$, $ \tilde{I}_i$ is an open neighborhood of $r_i$. Two arrows were added for each point $(r,i)$ in the interval $ \tilde{I}_i$: $\tilde r_i$ and its inverse arrow such that the source of $\tilde r_i$ is $(r,i)$ and its target is $(r,i+1)$.

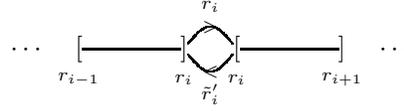
\begin{defn} A {\em generalized path} in the groupoid $G\ltimes X$ is a generalized map $I\overset{\e}{\gets}I_{S_n}\overset{\alpha}{\to}G\ltimes X$ such that
\begin{enumerate}
\item
$\e: I_{S_n}\to I$ on objects is the inclusion in each connected component, $\e(r,i)=r$ and on arrows it sends all arrows to identity arrows, $\e(\tilde r_i)=\id_{r}$
\item 
$\alpha: I_{S_n}\to G\ltimes X$ on objects is given by a map $\alpha_i :  I_i\to X$ in each connected component and on arrows is given by 
$\alpha(\tilde r_i)=(k_i, \alpha_i( r))$
satisfying the condition $k_i  \alpha_i( r)=\alpha_{i+1}( r)$ for all $r\in \tilde I_i$.
\end{enumerate}

\end{defn}

We denote $\mmm( I_{S_n},G\ltimes X)$ this space of maps from $I_{S_n}$ to $G\ltimes X$ with the compact open topology.

\subsubsection{Equivalence of generalized paths}
We will establish now an equivalence relation between the generalized maps defining our generalized paths which will allow us to give a groupoid structure to the space of generalized paths.

\begin{defn} Two  generalized paths $I\overset{\e}{\gets}I_{S_m}\overset{\alpha}{\to}G\ltimes X$ and $I\overset{\e'}{\gets}I_{S_{m'}}\overset{\beta}{\to}G\ltimes X$ are equivalent if there exist a subdivision $S_n$ and essential equivalences $u$ and $v$ such that the following diagram commutes up to natural transformations.
 $$
\xymatrix{ &
{I_{S_m}}\ar[dr]^{ \alpha}="0" \ar[dl]_{\e }="2"&\\
{I}&{I_{S_n}} \ar[u]_{u} \ar[d]^{v}&{G\ltimes X}\\
&{I_{S_{m'}}}\ar[ul]^{\e'}="3" \ar[ur]_{ \beta}="1"&
\ar@{}"0";"1"|(.4){\,}="7"
\ar@{}"0";"1"|(.6){\,}="8"
\ar@{}"7" ;"8"_{\sim_{}}
\ar@{}"2";"3"|(.4){\,}="5"
\ar@{}"2";"3"|(.6){\,}="6"
\ar@{}"5" ;"6"^{\sim}
}
$$
\end{defn}
Since $G$ is discrete, the condition $\alpha u \sim \beta v$ guarantees the existence of a natural transformation $T:\bigsqcup_{i=1}^{n} I_i\to G\times X$ such that $T(r,i)=(g_i, \alpha_i(r))$ with $\beta_i(r)=g_i\alpha_i(t)$. By naturality of the transformation we have that the following diagram commutes for all $r\in \tilde I_i$

$$\xymatrix{\alpha_i(r, i)\ar[r]^-{g_{i}}\ar[d]_{k_{i}}&\beta_i(r, i)\ar[d]^{k'_{i}}\\ \alpha_{i+1}(r, i)\ar[r]^{g_{i+1}}&\beta_{i+1}(r, i)}$$
therefore $k'_i=g_{i+1} k_i {g_i}^{-1}$ for all $i=1, \ldots, n-1$.

\begin{remark} Two generalized paths are equivalent if there exists a common subdivision $S_n$ and $g_i\in G$  such that $\beta_i(r)=g_i\alpha_i(r)$ for all $i=1, \ldots, n$ and $k'_i=g_{i+1} k_i {g_i}^{-1}$ for all $i=1, \ldots, n-1$.

\end{remark}
Then we have a translation groupoid $G^n\ltimes \mmm( I_{S_n},G\ltimes X)$ given by this
action of $G^n$ on the space $\mmm( I_{S_n},G\ltimes X)$. Source and target are given by $$s(  (g_1, \dots, g_{n}),(\alpha_1, \dots, \alpha_n, k_1,  \dots, k_{n-1}))=(\alpha_1,  \dots, \alpha_n, k_1,  \dots, k_{n-1})$$ and $$t(  (g_1,\dots, g_{n}),(\alpha_1,  \dots, \alpha_n, k_1,  \dots, k_{n-1}))=(g_1\alpha_1, \dots, g_n\alpha_n, g_{2}k_1{g^{-1}_1},  \dots, g_{n}k_{n-1}{g^{-1}_{n-1}}).$$

\subsubsection{Colimit construction}
In order to account for all possible subdivisions, we will consider the colimit of the groupoids $G^n\ltimes\mmm( I_{S_n},G\ltimes X)$ over a partially ordered set that we describe next.

We define the category $\mathcal{C}_I$ as the category with objects the ordered tuples $S_n=\{0=r_0\le r_1\le\cdots\le r_n=1\}$ with an open cover of $I=[0,1]$ given by connected intervals  $\{I_i \mid 1\leq i \leq n\}$. We require that:
\begin{enumerate}
\item  $[r_{i-1},r_{i}]\subseteq I_i$ and $I_i\cap \{r_0,r_1,\ldots,r_n\}=\{r_{i-1},r_i\}$, which is one point if $r_{i-1}=r_i$ and two points if $r_{i-1}<r_i$.
\item
\begin{itemize}
\item[2a)] If $r_{k-2}<r_{k-1}=r_k = \ldots = r_l<r_{l+1}$ then we require that  $I_{k} = I_{k+1} = \ldots = I_{l} \subseteq I_{k-1} \cap I_{l+1}$. 
\item[2b)] If $0=r_{0}=r_{1}= \ldots = r_k<r_{k+1}$ then we require that  $I_{1} = I_{2} = \ldots = I_{k} \subseteq I_{k+1} $.
\item[2c)] If $r_{k-1}<r_{k}=r_{k+1}= \ldots = r_n=1$ then we require that  $I_{k+1} = I_{k+2} = \ldots = I_{n} \subseteq I_{k}$.
\end{itemize}
\end{enumerate}

We have a morphism from $\left ( \{r_0\le r_1\le\cdots\le r_n\}, \{I_i\} \right )$ to $\left ( \{t_0\le t_1\le\cdots\le t_m\},\{\widetilde{I}_j\}\right ) $ if 
\begin{enumerate}
\item[I)] $\{r_0,r_1,\cdots,r_n\} \supseteq \{t_0,t_1,\cdots,t_m\}$. 
\item[II)] The multiplicity of repeated elements decreases, i.e. for every $i$, $| \{ j \mid r_j = r_i \} | \geq | \{ j \mid t_j = r_i \} |$.
\item[III)]  The open cover $\{I_i\}$ is a refinement of the open cover $\{\widetilde{I}_j\}$ in the following way: 
\begin{itemize}
\item[IIIa)]  For each closed interval $[r_{i-1},r_i]$ with non-empty interior there is  a unique $[t_{j-1},t_j]$ with $[r_{i-1},r_i] \subseteq [t_{j-1},t_j]$ and we have

$$\xymatrix{I_i \ar[r]^-{\subseteq} & \widetilde{I_j}\\ [r_{i-1},r_i] \ar[u]^{\subseteq} \ar[r]^{\subseteq}& [t_{j-1},t_j] \ar[u]^{\subseteq} }$$

\item[IIIb)] If there is a repeated element in the $\{t_0\le t_1\le\cdots\le t_m\}$,  $t_{j-1}=t_j$, it is also a repeated element of $\{r_0\le r_1\le\cdots\le r_n\}$,  $r_{i-1}=r_i$. We require $I_i \subseteq \widetilde{I}_j$.
\end{itemize}
\end{enumerate}

The morphisms are generated (as a category) by the set of morphisms:
\begin{enumerate}
\item Eliminating a point from the subdivision \\$\{0=r_0\le r_1\le\cdots\le r_i \le \cdots \le r_n=1\}$:
$$
d_i : \left(\{r_0\le  \cdots \le r_i \le \cdots \le r_n\},\{I_i\}\right ) \rightarrow \left(\{r_0\le\cdots\le \widehat{r_i} \le \cdots \le r_n\},\{\widetilde{I}_j\}\right)
$$
where $d_i$  drops the $i$-th element and concatenates the  consecutive intervals $I_i,I_{i+1}$, i.e: $\widetilde{I}_j=I_j$ for $j=0,\dots,i-1$, $\widetilde{I}_i=I_i \cup I_{i+1}$ and $\widetilde{I}_j=I_j$ for $j=i+1,\dots,n$.
\item Enlarging the intervals without changing the points of the subdivision $\{0=r_0\le r_1\le  \cdots \le r_n=1\}$:
$$
u: \left(\{r_0\le  \cdots \le r_n\},\{I_i\}\right ) \rightarrow \left(\{r_0\le
\cdots \le r_n\},\{\widetilde{I}_i\}\right)
$$
when $I_i \subseteq \widetilde{I}_i$.
\end{enumerate}

We call $\mathcal{C}_I$ the category of subdivisions of $I$ which is a cofiltered category, which boils down to the fact that for two subdivisions there is a common refinement.

For every morphism, there is a continuous map given by concatenation and inclusion
$$
\bigsqcup_{i} I_i \rightarrow \bigsqcup_{j} \widetilde{I}_j.
$$
To the morphism $d_i :S_n \rightarrow S_{n-1}$ we assign the functor  ${d_i}_* : I_{S_n} \rightarrow I_{S_{n-1}}$ that on objects concatenates  $I_i \cup I_{i+1}$ and on morphisms sends $\tilde r_i$ and its inverse arrow $\tilde r'_i$ to the identity arrow on $(r,i)$. Similarly for $ u : S_n \rightarrow S_n$, there is a functor $u_* : I_{S_n} \rightarrow I_{S_n}$ given by inclusion at the level of objects and morphisms. This gives a functor from $\mathcal{C}_I \rightarrow  \Grd$. We can obtain a contravariant functor $\psi$ from $\mathcal{C}_I^{op}$ to topological spaces that on objects sends $S_n$ to $\mmm( I_{S_n} , G\ltimes X)$ and on morphisms sends $d_i:S_n \rightarrow S_{n-1}$ to the morphism $d_i^* : \mmm( I_{S_{n-1}} , G\ltimes X) \rightarrow \mmm( I_{S_{n}}, G\ltimes X)$ given by  taking $\alpha \in \mmm ( I_{S_{n-1}}, G\ltimes X)$ represented by $(\alpha_1, \dots, \alpha_{n-1}, k_1,  \dots, k_{n-2})$ and sending it to $ 
(\alpha_1, \dots,\alpha_i\vert_{I_i},\alpha_i\vert_{I_{i+1}},\ldots, \alpha_{n-1}, k_1,  \dots,k_{i-1},id,k_{i},k_{i+1},\ldots, k_{n-2})$, i.e. taking  $\alpha_i : I_i \cup I_{i+1} \rightarrow X$ to the restrictions to $I_i$ and $I_{i+1}$. Similarly $u^*: \mmm( I_{S_{n}} , G\ltimes X) \rightarrow \mmm( I_{S_{n}}, G\ltimes X)$ is just restriction of all the paths: taking $\alpha \in \mmm ( I_{S_{n}}, G\ltimes X)$ represented by $(\alpha_1, \dots, \alpha_{n}, k_1,  \dots, k_{n-1})$ and sending it to $ 
(\alpha_1\vert_{I_1}, \dots, \alpha_{n} \vert_{I_n}, k_1, \ldots, k_{n-1})$.

We have an action of $G^n$ on the space $\mmm( I_{S_{n}} , G\ltimes X)$ given by 
$$
(g_1,\dots, g_{n})\cdot(\alpha_1,  \dots, \alpha_n, k_1,  \dots, k_{n-1})=(g_1\alpha_1, \dots, g_n\alpha_n, g_{2}k_1{g^{-1}_1},  \dots, g_{n}k_{n-1}{g^{-1}_{n-1}}).
$$ 
The map 
$
d_i^* : \mmm ( I_{S_{n-1}}, G\ltimes X) \rightarrow \mmm ( I_{S_{n}}, G\ltimes X)
$
is equivariant with respect to the map
$
\sigma_i : G^{n-1} \rightarrow G^n
$
given by $\sigma_i(g_1,\ldots,g_{n-1}) = (g_1,\ldots,g_i,g_i,g_{i+1},\ldots,g_{n-1})$. This means that
\begin{multline*}
\sigma_i(g_1,\ldots,g_{n-1}) \cdot d_i^*(\alpha_1,\ldots,\alpha_{n-1},k_1,\ldots,k_{n-2}) \\ = d_i^*((g_1,\ldots,g_{n-1})\cdot (\alpha_1,\ldots,\alpha_{n-1},k_1,\ldots,k_{n-2}))
\end{multline*}
This is  because $(g_1,\dots,g_i,g_i,g_{i+1},\ldots,g_{n-1})$ acting on 
$$
(\alpha_1, \dots,\alpha_i\vert_{I_i},\alpha_i\vert_{I_{i+1}},\ldots, \alpha_{n-1}, k_1,  \dots,k_{i-1},id,k_{i},k_{i+1},\ldots, k_{n-2})
$$ 
is equal in the first part to
$$
(g_1\alpha_1, \dots,g_i\alpha_i\vert_{I_i},g_i\alpha_i\vert_{I_{i+1}},\ldots, g_{n-1}\alpha_{n-1})
$$
and in the second part to
$$
( g_{2}k_1{g^{-1}_1},  \dots,g_{i}k_{i-1}{g^{-1}_{i-1}},g_i id g_i^{-1},g_{i+1}k_{i}g_i^{-1},\ldots, g_{n-1}k_{n-2}g_{n-2}^{-1})
$$ 
which is
$$
(g_{2}k_1{g^{-1}_1},  \dots,g_{i}k_{i-1}{g^{-1}_{i-1}},id ,g_{i+1}k_{i}g_i^{-1},\ldots, g_{n-1}k_{n-2}g_{n-2}^{-1})
$$
this is precisely
$$
d_i^*((g_1,\ldots,g_{n-1})\cdot (\alpha_1,\ldots,\alpha_{n-1},k_1,\ldots,k_{n-2})).
$$
Similarly the map $u^* : \mmm ( I_{S_{n}}, G\ltimes X) \rightarrow \mmm ( I_{S_{n}}, G\ltimes X)$ is equivariant with respect to the identity map $G^n \rightarrow G^n$.

Therefore we have a contravariant functor from $\mathcal{C}_I$ to the category of translation groupoids that on objects sends $S_n$ to $G^n \ltimes \mmm ( I_{S_n}, G\ltimes X)$ and on morphisms sends $d_i : S_n \rightarrow S_{n-1}$ to the functor $(d_i^*,\sigma_i)$ and $u :  S_n \rightarrow S_n$ to the functor $(u^*,id)$, formally we have a (covariant) functor
$
\xymatrix{
\varPhi : \mathcal{C}_I^{op} \ar[r] & \TrG.}
$

We consider now the (filtered) colimit of $\varPhi$,
$$
P= \colim_{\mathcal{C}_I^{op}} \varPhi
$$
given by an object $P\in \TrG$ together with  morphisms from $\mmm ( I_{S_n}, G\ltimes X)$ for each $S_n$ such that  for each morphism  the following diagrams commute:

For $d_i$:
$$
\begin{tikzcd}
G^n\ltimes \mmm ( I_{S_n}, G\ltimes X) \arrow[]{dr}{} \arrow[]{dd}{{\sigma_i \ltimes d_i^*}} &   \\
&P    \\
G^{n-1}\ltimes \mmm ( I_{S_{n-1}} , G\ltimes X) \arrow{ru}{}& 
\end{tikzcd}
$$
For $u$:
$$
\begin{tikzcd}
G^n\ltimes \mmm ( I_{S_n}, G\ltimes X) \arrow[]{dr}{} \arrow[]{dd}{{id \ltimes u^*}} &   \\
&P    \\
G^{n}\ltimes \mmm ( I_{S_{n}} , G\ltimes X) \arrow{ru}{}& 
\end{tikzcd}
$$

Moreover, $P=\colim \varPhi$ has the following universal property. Given another translation groupoid $W$ with functors from $G^n\ltimes \mmm ( I_{S_n}, G\ltimes X)$ that are compatible, such functors factor uniquely through the colimit $P$ as shown in the following diagrams:
$$
\begin{tikzcd}
G^n\ltimes \mmm ( I_{S_n}, G\ltimes X) \arrow[]{dr}{} \arrow[bend left]{drr}{} \arrow[]{dd}{{\sigma_i \ltimes d_i^*}} & &  \\
&P  \arrow[dashed]{r} & W\\
G^{n-1}\ltimes \mmm ( I_{S_{n-1}}, G\ltimes X) \arrow{ru}{}\arrow[bend right]{urr}{}& &
\end{tikzcd}
$$
$$
\begin{tikzcd}
G^n\ltimes \mmm ( I_{S_n}, G\ltimes X) \arrow[]{dr}{} \arrow[bend left]{drr}{} \arrow[]{dd}{{id \ltimes u^*}} & &  \\
&P  \arrow[dashed]{r} & W\\
G^{n}\ltimes \mmm ( I_{S_{n}}, G\ltimes X) \arrow{ru}{}\arrow[bend right]{urr}{}& &
\end{tikzcd}
$$

\begin{defn} The {\em path groupoid } $P(G\ltimes X)$ of the translation groupoid $G\ltimes X$ is 
$$P(G\ltimes X)= \colim_{\mathcal{C}_I^{op}} \varPhi$$ where $\varPhi :  \mathcal{C}_I^{op} \to  \TrG$ is as above.
\end{defn}

We are ready now to give an explicit construction of the groupoid $P=P(G\ltimes X)$ by using the constructions of colimits in the category of topological spaces $\Top$ and in the category of groups $\Grp$.

The colimit of the contravariant functor $\psi:\mathcal{C}_I^{op}\to \Top$ is a topological space $M=\colim \psi$  such that
$$M=({\coprod_{\mathcal{C}_I} \mmm( I_{S_n},G\ltimes X)})/\sim$$
where $\sim$ is the equivalence relation generated by $\alpha\sim d_i^*(\alpha)$ for all $S_n$  and $d_i:S_n \rightarrow S_{n-1}$ and $\alpha\sim u^*(\alpha)$ for all $S_n$  and $u:S_n \rightarrow S_{n}$.

This topological space $M=\colim \psi$ will be the space of objects of the path groupoid $P$. To construct the space of arrows of the path groupoid, we consider now a colimit in the category of groups.

Consider the functor $\varphi : \mathcal{C}_I^{op}\to \Grp$ which sends $S_n$ to $G^n$ and on morphisms sends $u: S_n \rightarrow S_n$ to the identity $G^n \rightarrow G^n$ and  $d_i:S_n \rightarrow S_{n-1}$ to the morphism $\sigma_i : G^{n-1} \rightarrow G^n$
given by $\sigma_i(g_1,\ldots,g_{n-1}) = (g_1,\ldots,g_i,g_i,g_{i+1},\ldots,g_{n-1})$. 

The colimit of $\varphi$ is a group $H=\colim \varphi$  such that
$$H=({\coprod_{\mathcal{C}_I} G^n})/\sim$$
where $\sim$ is generated by $(g_1,\ldots,g_{n-1}) \sim (g_1,\ldots,g_i,g_i,g_{i+1},\ldots,g_{n-1})$. This group $H$ is discrete and acts on the topological space $M$ constructed above.

We can describe  now explicitly the object and arrow spaces of the path groupoid $P=P(G\ltimes X)$ in $\TrG$:
$$P_0=M= \colim \psi= {\coprod_{\mathcal{C}_I} \mmm( I_{S_n},G\ltimes X)}/\sim$$
and $$P_1= H\times M= \colim \varphi \times \colim \psi =\left(({\coprod_{\mathcal{C}_I} G^n})/\sim \right)\times \left(({\coprod_{\mathcal{C}_I} \mmm( I_{S_n},G\ltimes X)})/\sim\right)$$ which we endow with the inductive topology.

\begin{remark} Let $G$ be a discrete group acting on $X$. The  path groupoid of $G\ltimes X$ is the translation groupoid $$P=P(G\ltimes X)=H \ltimes M.$$ 
\end{remark}
We will  show that this path groupoid $P=\colim \varPhi$ described above is actually equivalent to the translation groupoid $G\ltimes X^I$. In order to give an explicit characterization of the equivalence of categories, we will introduce some auxiliary groupoids which in turn will relate to the idea introduced in \cite{colman2011} of multiple $G$-paths.


\subsection{Multiple $G$-paths} We will provide  now another description of the path groupoid in terms of equivariant generalized maps. We will see that for each generalized path $(\e,\alpha)$, its equivalence class  $[I\overset{\e}{\gets}I_{S_n}\overset{\alpha}{\to}G\ltimes X]$  contains a representative  in $\MTrG$ of the form:
$$I\overset{\delta}{\gets}G\ltimes Y\overset{\phi}{\to}G\ltimes X$$
where $G\ltimes Y$ is a translation groupoid.

Given a generalized path $I\overset{\e}{\gets}I_{S_n}\overset{\alpha}{\to}G\ltimes X$,  we will construct a space $Y=Y_{\alpha}$ such that $G\ltimes Y$ is Morita equivalent to $I_{S_n}$, and maps $\delta: G\ltimes Y\overset{}{\to}I$ and $\phi: G\ltimes Y\overset{}{\to}G\ltimes X$ such that $(\delta, \phi)$ is $2$-isomorphic to the given $G$-path $(\e, \alpha)$.

\subsubsection{Construction of  $G\ltimes Y_{\alpha}$} Let $\alpha=(\alpha_1, \ldots, \alpha_n, k_1, \ldots, k_{n-1})$.
Consider the product space $$G\times (I_{S_n})_0=\{(g,(r,i))|\;g\in G, (r,i)\in I_i\}$$ and the following identifications for all $r\in \tilde I_i$:  $$(g,(r,i+1))\sim(k_i^{-1}g,(r,i))$$ where $\alpha(\tilde r_i)=(k_i,\alpha_i(r))$.

 Now $Y_{\alpha}$ is defined as the quotient space:
$$Y_{\alpha}=\{[(g,(r,i))] |\; (g,(r,i))\in G\times (I_{S_n})_0 \mbox{ and }(g,(r,i+1))\sim(k_i^{-1}g,(r,i)) \mbox{ for all }r\in \tilde I_i\}.$$

Observe that the space $Y_{\alpha}$ depends on $\alpha$ in the sense that it is given by the subdivision $S_n$ and the group elements  $k_1, \ldots, k_{n-1}$, but it is independent of the actual pieces of the path $\alpha_1, \ldots, \alpha_n$.

The action of $G$ on  $Y_{\alpha}$ is given by  the  multiplication in the group $h([g,(r,i)])=[gh^{-1},(r,i)]$. 

We can consider then the translation groupoid $G\ltimes Y_{\alpha}$ where the source and target are given by the maps $s(h, [g,(r,i)])=[g,(r,i)]$ and $t(h, [g,(r,i)])=[gh^{-1},(r,i)]$. 

\subsubsection{Morita equivalence  $I_{S_n}\sim_M G\ltimes Y_{\alpha}$} We will show now that the translation groupoid constructed above is Morita equivalent to the groupoid $I_{S_n}$.
Let $\nu: I_{S_n} \to G\ltimes Y_{\alpha}$ be the morphism defined by 
$\nu((r,i))=[e, (r,i)]$ on objects and $\nu(\tilde r_i)=(k_i, [e, (r,i)])$ on arrows for all $r\in \tilde I_i$.  
The open map $\nu$ is essentially surjective since
$s\pi_1: G\times (I_{S_n})_0\to Y_{\alpha}=(G\times (I_{S_n})_0)/\sim$
is the quotient projection. It is also fully faithful since $(I_{S_n})_1$ is given by the pullback of the following maps
\[\xymatrix{
& G\times Y_{\alpha} \ar[d]^{(s,t)} \\
(I_{S_n})_0\times (I_{S_n})_0 \ar[r]^{\nu \times \nu} & Y_{\alpha}\times Y_{\alpha}}\]

Therefore given a groupoid $I_{S_n}$, we can construct another groupoid $Y_{\alpha}$ for each set of elements $k_1, \ldots, k_{n-1}$ such that   $I_{S_n}$ is Morita equivalent to $G\ltimes Y_{\alpha}$.

\subsubsection{The $2$-isomorphism $(\e, \alpha)\Rightarrow(\delta, \phi)$}
We will define now the maps $\delta$ and $\phi$ to obtain the generalized map $I\overset{\delta}{\gets}G\ltimes Y_{\alpha}\overset{\phi}{\to}G\ltimes X$ being $2$-isomorphic to the given generalized path $(\e, \alpha)$.  

We define $\phi([g,(r,i)])=g^{-1}\alpha_i(r)$ on objects and $\phi(h, [g,(r,i)])=(h, g^{-1}\alpha_i(r))$ on arrows. Moreover, the morphism $\phi$ is $G$-equivariant in the ordinary sense (the group homomorphism is the identity).

The essential equivalence $\delta: G\ltimes Y_{\alpha}\overset{}{\to}I$ is given by projection on both objects and arrows, $\delta(h, [g,(r,i)])= r$. Both morphisms $\phi$ and $\delta$ are well defined and $\delta$ is open, surjective on objects and fully faithful.

We have that the following diagram is commutative:
$$
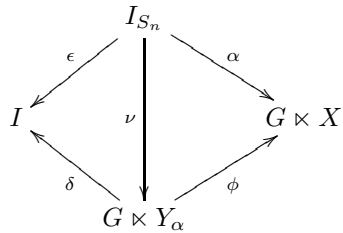
\xymatrix{ &I_{S_n}\ar[dr]^{\alpha}\ar[dl]_{\e}\ar[dd]_{\nu}&\\
I& &{G\ltimes X}\\
&{G\ltimes Y_{\alpha}}\ar[ul]^{\delta}="3" \ar[ur]_{\phi}="1"&
}
$$ 
since $\phi\nu((r,i))=\phi([e, (r,i)])=\alpha_i(r_i)$ and $\phi\nu(\tilde r_i)=\phi(k_i, [e, (r,i)])=(k_i, \alpha_i(r))$ for all $r\in \tilde I_i$.

Therefore there is a $2$-isomorphism between the generalized map $I\overset{\delta}{\gets}G\ltimes Y_{\alpha}\overset{\phi}{\to}G\ltimes X$  and the generalized path $I\overset{\e}{\gets}I_{S_n}\overset{\alpha}{\to}G\ltimes X$.

Observe that the identifications we have made in the quotient to obtain the space $Y_{\alpha}$ determine a gluing of the segments $I_i$ at the different levels of $G\times (I_{S_n})_0$ to obtain copies of the entire interval $I=[0,1]$. This gluing is determined by the group elements $k_1, \ldots, k_{n-1}$.

To define the map $\phi$ from the groupoid $G\ltimes Y_{\alpha}$ associated to the generalized path $\alpha$, we are concatenating the different pieces $\alpha_i$ in these different levels by multiplying by the correct group element to obtain an honest path in $X$.

\subsubsection{The homeomorphism $\gamma :Y_{\alpha}\to G\times I$}
For each map $\alpha:I_{S_n}\to G\ltimes X$, let's show now that the space $Y_{\alpha}$ we just constructed is $G$-equivariantly homeomorphic to the space $G\times I$, where the action on the latter is determined by the action of $G$ on $Y_{\alpha}$ given by $h[g,(r,i)]=[gh^{-1},(r,i)]$. We have that the action on $G\times I$  is given by
$$G\times (G\times I)\to G\times I$$
$$h,(g,r)=(gh^{-1},r).$$
We define the homeomorphism $\gamma_{}:Y_{\alpha}\to G\times I$ as
$$\gamma_{}([g,(r,i)])=((k_{i-1}\cdots k_1)^{-1}g, r)$$ for $i=1, \cdots, n$.   
The morphism $\gamma_{}$ depends only on $S_n$ and $k_1, \cdots, k_{n-1}$ and is independent on the actual paths $\alpha_1, \cdots, \alpha_{n}$.
The inverse morphism $\gamma^{-1}:G\times I\to Y_{\alpha}$ is given by 
$$\gamma^{-1}(h,r)=[k_{i-1}\cdots k_1h,(r,i)]$$ if $r\in I_i$. Moreover, 
we have that the homeomorphism $\gamma_{}$ is $G$-equivariant by construction.

\begin{defn} A {\em multiple G-path} in the groupoid $G\ltimes X$ is a generalized map $I\overset{}{\gets}G\ltimes (G\times I)\overset{\sigma}{\to}G\ltimes X$ where $\sigma$ is a $G$-equivariant map in the ordinary sense.
\end{defn}
\subsubsection{Equivalence of multiple $G$-paths}
 Given two multiple $G$-paths $I\overset{}{\gets}G\ltimes (G\times I)\overset{\sigma}{\to}G\ltimes X$ and $I\overset{}{\gets}G\ltimes (G\times I)\overset{\tau}{\to}G\ltimes X$,  they are equivalent if there exists a subdivision $S_n$ and $k_1,\cdots, k_{n-1}\in G$ such that the following diagram commutes  up to natural transformations
 $$
\xymatrix{ &G\ltimes (G\times I)\ar[dr]^{\sigma}\ar[dl]_{p}&\\
I& I_{S_n}\ar[u]_{\nu} \ar[d]^{\eta} &{G\times X}\\
&G\ltimes (G\times I)\ar[ul]^{p} \ar[ur]_{\tau}&
}
$$
where $\nu=\nu_{k_1,\cdots, k_{n-1}}$ and $\eta=\eta_{k_1,\cdots, k_{n-1}}$.

 Since $p$ is an essential equivalence, we have that $\nu\sim \eta$ and then $\sigma\nu\sim \tau\nu$. That means that there exists a natural transformation $T:  (I_{S_n})_0\to G\times X$ such that $T(r,i)$ is an arrow between $\sigma\nu(r,i)= \sigma( (k_{i-1}\cdots k_1)^{-1},r)$ and $\tau( (k_{i-1}\cdots k_1)^{-1},r)$. Therefore we have that the multiple $G$-paths are equivalent if there exists a subdivision $S_n$, $k_1,\cdots, k_{n-1}\in G$ and $g_1,\cdots, g_n\in G$ such that
$${g_i} \sigma((k_{i-1}\cdots k_1)^{-1},r)=\tau((k_{i-1}\cdots k_1)^{-1},r)\mbox{ if }r\in I_i.$$
 Since $\sigma$ is equivariant, we have that 
 $${g_i} (k_{i-1}\cdots k_1)\sigma(e,r)= (k_{i-1}\cdots k_1)\tau(e,r)\mbox{ if }r\in I_i.$$
 For $i=1$ this means that there exists $g_1\in G$ such that $\tau(e,r)= g_1\sigma(e,r)$. Since the interval $e\times I$ is connected, we have that $g_i=(k_{i-1}\cdots k_1) g_1 (k_{i-1}\cdots k_1)^{-1}$ for all $i=1,\ldots, n$. In other words, all other $g_i$, $i=2,\ldots, n$ are determined by $g_1$. Once that we have a  group element $g_1\in G$ that makes $\tau(e,r)= g_1\sigma(e,r)$ in the first piece of the interval, $r\in[0,r_1]$, then all the other pieces of the interval coming from the subdivision $S_n$ will also coincide since for all $r\in I_i$
 $${g_i} (k_{i-1}\cdots k_1)\sigma(e,r)= (k_{i-1}\cdots k_1)\tau(e,r)$$ and 
 $$g_i=(k_{i-1}\cdots k_1) g_1 (k_{i-1}\cdots k_1)^{-1}$$
 Then $${(k_{i-1}\cdots k_1) g_1 (k_{i-1}\cdots k_1)^{-1}} (k_{i-1}\cdots k_1)\sigma(e,r)= (k_{i-1}\cdots k_1)\tau(e,r)$$ which implies that $g_1\sigma(e,r)= \tau(e,r)$ for all $r\in I$.

 \begin{prop}
 Two multiple $G$-paths $\sigma$ and $\tau$ are equivalent if there exists $g\in G$ such that 
$$g\sigma(e,r)= \tau(e,r).$$
\end{prop}
Then we have the group $G$ acting now on the space of equivariant maps $G\mmm(G\times I, X)$. Let $P'=G\ltimes G\mmm(G\times I, X)$ be the {\em multiple $G$-path groupoid}.

Since $\sigma(g,r)=g\sigma(e,r)$, we observe that a multiple $G$-path is determined by the honest path $\beta: I\to X$ given by $\beta(r)=\sigma(e,r)$.
Conversely, any path $\beta: I\to X$ can be made into a multiple $G$-path by putting 
$\sigma(g,r)=g\beta(r).$ Consider the translation groupoid of honest paths, given by the obvious action of $G$ on $X$. Let $P''= G\ltimes X^I$ where $X^I=\mmm(I,X)$.

We will prove next that all three characterizations of the path groupoid, as generalized paths, as multiple $G$-paths and as honest paths are equivalent.

\subsection{Equivalence of the different models for path groupoids }
Recall the definition of the path groupoid and the other two characterizations  introduced in the previous section. 
\begin{enumerate}
  \item The groupoid  $P=\colim \varphi \ltimes \colim \psi$ where  $M=\colim \psi$  is the space of classes of generalized paths.
  \item The groupoid $P'=G\ltimes G\mmm(G\times I, X)$ where $G\mmm(G\times I, X)$ is the space of $G$-equivariant maps.
  \item 
The groupoid $P''=G\ltimes X^I$ where $X^I$ is the free path space.
\end{enumerate}

\subsubsection{The equivalence of categories $\chi: P=\colim \varphi \ltimes \colim \psi\to P'=G\ltimes G\mmm(G\times I, X)$.} 
Recall that $M=\colim \psi$  is the space of classes of generalized paths, i.e. 
$$M=({\coprod_{\mathcal{C}_I} \mmm( I_{S_n},G\ltimes X)})/\sim$$
where $\sim$ is the equivalence relation generated by $\alpha\sim d_i^*(\alpha)$ for all $S_n$  and $d_i:S_n \rightarrow S_{n-1}$ and $\alpha\sim u^*(\alpha)$ for all $S_n$  and $u:S_n \rightarrow S_{n}$. We will use the same notation $(\alpha_1, \cdots, \alpha_n, k_1, \cdots, k_{n-1})$ to denote the elements in $M$.

The idea is to complete each piece $\alpha_i$ of the  generalized path $\alpha=(\alpha_1, \cdots, \alpha_n, k_1, \cdots, k_{n-1})$ to have the entire branch $\sigma_i$ of a multiple $G$-path $\sigma$.

\paragraph{The functor  $\chi: \colim \varphi \ltimes \colim \psi\to G\ltimes G\mmm(G\times I, X)$.}
 Given a generalized path $\alpha=(\alpha_1, \cdots, \alpha_n, k_1, \cdots, k_{n-1})$ for the subdivision $S_n$ of the interval $I$  we can define (as in the previous section)
 \begin{enumerate}
 \item a space $Y_{\alpha}=\{[(g,(r,i))] |\; (g,(r,i))\in G\times (I_{S_n})_0\}$ with the relation $(g,(r,i+1))\sim(k_i^{-1}g,(r,i))$ for all $r\in \tilde I_i$,
  \item a homeomorphism $\gamma_{\alpha}: G\ltimes Y_{\alpha} \to G\ltimes (G\times I)$,
    \item an essential equivalence $\nu_{\alpha}: I_{S_n} \to G\ltimes Y_{\alpha}$ and 
\item a generalized map $I\overset{\delta}{\gets}G\ltimes Y_{\alpha}\overset{\phi_{\alpha}}{\longrightarrow}G\ltimes X$ such that  $(\e, \alpha)\Rightarrow(\delta, \phi_{\alpha})$. 
    \end{enumerate}
    
We define $\chi: \colim \varphi \ltimes M \to G\ltimes G\mmm(G\times I, X)$ as $\chi(\alpha)=\phi_{\alpha}{\gamma_{\alpha}}^{-1}$ on objects and $\chi(g_1, \cdots, g_{n}, \alpha)=(g_1,\phi_{\alpha}{\gamma_{\alpha}}^{-1})$ on arrows.

Then $\chi(\alpha)(g,r)=\phi_{\alpha}{\gamma_{\alpha}}^{-1}(g,r)=\phi_{\alpha}[k_{i-1}\cdots k_1g,(r,i)]=(k_{i-1}\cdots k_1g)^{-1}\alpha_i(r)$ if $r\in I_i$.
We are sending each generalized path $\alpha=(\alpha_1, \cdots, \alpha_n, k_1, \cdots, k_{n-1})$ into the multiple $G$-path $\sigma$ given by
$$\sigma(g,r)=g^{-1} (k_{i-1}\cdots k_1)^{-1}\alpha_i(r)\mbox{ if }r\in I_i.$$
In particular, we have that the branch $\sigma_e$ corresponding to the interval $e\times I$ is given by the following concatenation:
$$\sigma(e,r)=\alpha_1(r)*{k_1}^{-1} \alpha_2(r)*(k_2k_1)^{-1} \alpha_3(r)*\cdots *  (k_{n-1}\cdots k_1)^{-1}\alpha_n(r)
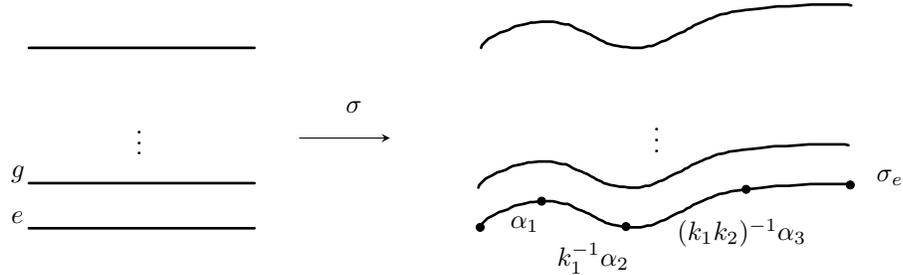
$$

\begin{tikzpicture}[line cap=round,line join=round,>=stealth,x=0.6cm,y=0.6cm]
\clip(-3.,-6.) rectangle (18.,2.);
\draw [line width=1pt] (-2.,-4.)-- (3.,-4.);
\draw [line width=1pt] (-2.,-3.)-- (3.,-3.);
\draw [line width=1pt] (-2.,0.)-- (3.,0.);
\draw [->] (4.,-2.) -- (6.,-2.);
\draw [line width=1pt] (8.02,0.)-- (8.04,0.06)-- (8.1,0.12)-- (8.18,0.2)-- (8.26,0.24)-- (8.34,0.28)-- (8.44,0.34)-- (8.52,0.36)-- (8.58,0.38)-- (8.64,0.4)-- (8.7,0.44)-- (8.78,0.48)-- (8.88,0.5)-- (8.98,0.52)-- (9.06,0.54)-- (9.12,0.56)-- (9.2,0.56)-- (9.36,0.58)-- (9.44,0.58)-- (9.52,0.58)-- (9.62,0.58)-- (9.72,0.56)-- (9.78,0.54)-- (9.88,0.52)-- (9.96,0.52)-- (10.04,0.48)-- (10.1,0.44)-- (10.18,0.4)-- (10.24,0.38)-- (10.3,0.36)-- (10.38,0.32)-- (10.44,0.3)-- (10.5,0.26)-- (10.58,0.22)-- (10.64,0.18)-- (10.74,0.14)-- (10.8,0.1)-- (10.88,0.08)-- (10.94,0.06)-- (11.06,0.04)-- (11.16,0.02)-- (11.24,0.02)-- (11.36,0.)-- (11.48,0.)-- (11.58,0.)-- (11.66,0.02)-- (11.72,0.04)-- (11.8,0.04)-- (11.86,0.06)-- (11.94,0.1)-- (12.02,0.14)-- (12.1,0.16)-- (12.18,0.2)-- (12.24,0.24)-- (12.32,0.28)-- (12.4,0.32)-- (12.52,0.36)-- (12.6,0.42)-- (12.66,0.44)-- (12.74,0.5)-- (12.84,0.54)-- (12.98,0.6)-- (13.06,0.64)-- (13.14,0.66)-- (13.22,0.7)-- (13.32,0.74)-- (13.46,0.76)-- (13.58,0.8)-- (13.76,0.82)-- (13.88,0.84)-- (14.1,0.86)-- (14.28,0.88)-- (14.58,0.92)-- (14.84,0.92)-- (15.1,0.94)-- (15.34,0.96)-- (15.54,0.96)-- (15.7,0.96)-- (15.86,0.96)-- (15.96,0.96)-- (16.06,0.96)-- (16.14,0.96)-- (16.22,0.94);
\draw [line width=1pt] (8.,-3.98)-- (8.02,-3.92)-- (8.08,-3.86)-- (8.16,-3.78)-- (8.24,-3.74)-- (8.32,-3.7)-- (8.42,-3.64)-- (8.5,-3.62)-- (8.56,-3.6)-- (8.62,-3.58)-- (8.68,-3.54)-- (8.76,-3.5)-- (8.86,-3.48)-- (8.96,-3.46)-- (9.04,-3.44)-- (9.1,-3.42)-- (9.18,-3.42)-- (9.34,-3.4)-- (9.42,-3.4)-- (9.5,-3.4)-- (9.6,-3.4)-- (9.7,-3.42)-- (9.76,-3.44)-- (9.86,-3.46)-- (9.94,-3.46)-- (10.02,-3.5)-- (10.08,-3.54)-- (10.16,-3.58)-- (10.22,-3.6)-- (10.28,-3.62)-- (10.36,-3.66)-- (10.42,-3.68)-- (10.48,-3.72)-- (10.56,-3.76)-- (10.62,-3.8)-- (10.72,-3.84)-- (10.78,-3.88)-- (10.86,-3.9)-- (10.92,-3.92)-- (11.04,-3.94)-- (11.14,-3.96)-- (11.22,-3.96)-- (11.34,-3.98)-- (11.46,-3.98)-- (11.56,-3.98)-- (11.64,-3.96)-- (11.7,-3.94)-- (11.78,-3.94)-- (11.84,-3.92)-- (11.92,-3.88)-- (12.,-3.84)-- (12.08,-3.82)-- (12.16,-3.78)-- (12.22,-3.74)-- (12.3,-3.7)-- (12.38,-3.66)-- (12.5,-3.62)-- (12.58,-3.56)-- (12.64,-3.54)-- (12.72,-3.48)-- (12.82,-3.44)-- (12.96,-3.38)-- (13.04,-3.34)-- (13.12,-3.32)-- (13.2,-3.28)-- (13.3,-3.24)-- (13.44,-3.22)-- (13.56,-3.18)-- (13.74,-3.16)-- (13.86,-3.14)-- (14.08,-3.12)-- (14.26,-3.1)-- (14.56,-3.06)-- (14.82,-3.06)-- (15.08,-3.04)-- (15.32,-3.02)-- (15.52,-3.02)-- (15.68,-3.02)-- (15.84,-3.02)-- (15.94,-3.02)-- (16.04,-3.02)-- (16.12,-3.02)-- (16.2,-3.04);
\draw [line width=1pt] (7.96,-3.1)-- (7.98,-3.04)-- (8.04,-2.98)-- (8.12,-2.9)-- (8.2,-2.86)-- (8.28,-2.82)-- (8.38,-2.76)-- (8.46,-2.74)-- (8.52,-2.72)-- (8.58,-2.7)-- (8.64,-2.66)-- (8.72,-2.62)-- (8.82,-2.6)-- (8.92,-2.58)-- (9.,-2.56)-- (9.06,-2.54)-- (9.14,-2.54)-- (9.3,-2.52)-- (9.38,-2.52)-- (9.46,-2.52)-- (9.56,-2.52)-- (9.66,-2.54)-- (9.72,-2.56)-- (9.82,-2.58)-- (9.9,-2.58)-- (9.98,-2.62)-- (10.04,-2.66)-- (10.12,-2.7)-- (10.18,-2.72)-- (10.24,-2.74)-- (10.32,-2.78)-- (10.38,-2.8)-- (10.44,-2.84)-- (10.52,-2.88)-- (10.58,-2.92)-- (10.68,-2.96)-- (10.74,-3.)-- (10.82,-3.02)-- (10.88,-3.04)-- (11.,-3.06)-- (11.1,-3.08)-- (11.18,-3.08)-- (11.3,-3.1)-- (11.42,-3.1)-- (11.52,-3.1)-- (11.6,-3.08)-- (11.66,-3.06)-- (11.74,-3.06)-- (11.8,-3.04)-- (11.88,-3.)-- (11.96,-2.96)-- (12.04,-2.94)-- (12.12,-2.9)-- (12.18,-2.86)-- (12.26,-2.82)-- (12.34,-2.78)-- (12.46,-2.74)-- (12.54,-2.68)-- (12.6,-2.66)-- (12.68,-2.6)-- (12.78,-2.56)-- (12.92,-2.5)-- (13.,-2.46)-- (13.08,-2.44)-- (13.16,-2.4)-- (13.26,-2.36)-- (13.4,-2.34)-- (13.52,-2.3)-- (13.7,-2.28)-- (13.82,-2.26)-- (14.04,-2.24)-- (14.22,-2.22)-- (14.52,-2.18)-- (14.78,-2.18)-- (15.04,-2.16)-- (15.28,-2.14)-- (15.48,-2.14)-- (15.64,-2.14)-- (15.8,-2.14)-- (15.9,-2.14)-- (16.,-2.14)-- (16.08,-2.14)-- (16.16,-2.16);
\draw (0.16,-1.28) node[anchor=north west] {$\vdots$};
\draw (11.62,-1.22) node[anchor=north west] {$\vdots$};
\draw (4.82,-0.98) node[anchor=north west] {$\sigma$};
\draw (-2.6,-3.42) node[anchor=north west] {$e$};
\draw (-2.6,-2.4) node[anchor=north west] {$g$};
\draw (16.58,-2.5) node[anchor=north west] {$\sigma_{e}$};
\draw (8.46,-3.54) node[anchor=north west] {$\alpha_1$};
\draw (9.5,-4.12) node[anchor=north west] {$k_1^{-1}\alpha_2$};
\draw (12.16,-3.58) node[anchor=north west] {$(k_1k_2)^{-1}\alpha_3$};
\begin{scriptsize}
\draw [fill=black] (8.,-3.98) circle (1.5pt);
\draw [fill=black] (9.36,-3.4) circle (1.5pt);
\draw [fill=black] (11.229729729729732,-3.961621621621621) circle (1.5pt);
\draw [fill=black] (13.892459016393444,-3.1370491803278693) circle (1.5pt);
\draw [fill=black] (16.2,-3.04) circle (1.5pt);
\end{scriptsize}
\end{tikzpicture}

On arrows, we send $((g_1, \cdots, g_{n}), \alpha_1, \cdots, \alpha_n, k_1, \cdots, k_{n-1})\in \colim \varphi \times \colim \psi$ into the arrow $(g_1, \sigma_{\alpha})$ where $\sigma_{\alpha}$ is defined as before.

We will show next that $\chi$ is an equivariant map between translation groupoids where the group homomorphism is given by the projection on the first coordinate. 

Let $\alpha'=(g_1\alpha_1, \ldots, g_{n}\alpha_n, g_2k_1{g_1}^{-1}, \ldots, g_n k_{n-1}{g_{n-1}}^{-1})$, we have that
$$\chi(\alpha')=g_1\chi((\alpha_1, \ldots, \alpha_n, k_1, \cdots, k_{n-1}))$$
since $$\sigma_{\alpha'}(e,r)=g_1\alpha_1(r)*({g_2k_1{g_1}^{-1}})^{-1} g_2\alpha_2(r)*\ldots *  (g_nk_{n-1}{g_{n-1}}^{-1}\ldots g_2k_1{g_1}^{-1})^{-1}g_n\alpha_n(r)$$
$$=
g_1(\alpha_1(r)*{k_1}^{-1} \alpha_2(r)*(k_2k_1)^{-1} \alpha_3(r)*\cdots *  (k_{n-1}\cdots k_1)^{-1}\alpha_n(r)=g_1\sigma_{\alpha}(e,r).
$$

\paragraph{The functor  $\chi^{-1}:      G\ltimes G\mmm(G\times I, X) \to  \colim \varphi \ltimes \colim \psi  $.}

Consider the continuous functor given by $\chi^{-1}(\sigma)=\sigma|_{e\times I}\circ i_{e}$ on objects and $\chi^{-1}((g, \sigma))=(g, \sigma|_{e\times I}\circ i_{e})$ on arrows, where $i_e:I\to e\times I$ sends $r\in I$ to $(e,r)\in e\times I$. Recall that by our notation convention the right side means in both cases the class in the colimit. Note that the generalized path $\sigma|_{e\times I}\circ i_{e}$ corresponds to a subdivision $S_1$ with only one subinterval, that is, $\sigma|_{e\times I}\circ i_{e}$ is an honest path.

\paragraph{$\chi$ and  $\chi^{-1}$ are inverse functors up to natural transformation.}

The composition $\chi\circ\chi^{-1}: G\ltimes G\mmm(G\times I, X)\to G\ltimes G\mmm(G\times I, X)$ is the identity map since the generalized map $\alpha_{\sigma}$ associated to $\sigma$ has only one piece. On objects:	$$\chi\circ\chi^{-1}(\sigma)=\chi(\alpha_{\sigma})=\sigma_{\alpha_{\sigma}}$$ such that 
$\sigma_{\alpha_{\sigma}}(g,r)=g^{-1}\sigma(e,r)=	\sigma(g,r)$, then $\sigma_{\alpha_{\sigma}}=\sigma$. On arrows 
$$\chi\circ\chi^{-1}(g,\sigma)=\chi(g, \sigma_{\alpha_{\sigma}})=\chi(g,\sigma)=(g, \sigma).$$
We will prove next that the composition in the other direction is equivalent by a natural transformation to the identity. We have that $\chi^{-1}\circ \chi: \colim \varphi \ltimes \colim \psi\to \colim \varphi \ltimes \colim \psi$ sends each generalized path class $\alpha=(\alpha_1, \cdots, \alpha_n, k_1, \cdots, k_{n-1})$ to the  generalized path $\alpha_{\sigma_{\alpha}}$ where
$$\alpha_{\sigma_{\alpha}}(r)=\sigma_{\alpha}(e,r)=\alpha_1(r)*{k_1}^{-1} \alpha_2(r)*(k_2k_1)^{-1} \alpha_3(r)*\cdots *  (k_{n-1}\cdots k_1)^{-1}\alpha_n(r)
$$ and each arrow $((g_1, \cdots, g_{n}), \alpha_1, \cdots, \alpha_n, k_1, \cdots, k_{n-1})\in \colim \varphi \times \colim \psi$ to the arrow $(g_1,\alpha_{\sigma_{\alpha}})$.

There is a natural transformation $T:  \colim \psi \to \colim \varphi \times \colim \psi$ given by $$T(\alpha)=((\id, {k_1}^{-1}, (k_2k_1)^{-1},\ldots  (k_{n-1}\cdots k_1)^{-1}), (\alpha_1, \cdots, \alpha_n, k_1, \cdots, k_{n-1}))$$
which is an arrow between $\alpha$ and $\alpha_{\sigma_{\alpha}}$ since 
\begin{multline*}
(\id, {k_1}^{-1}, (k_2k_1)^{-1},\ldots,  (k_{n-1}\cdots k_1)^{-1})(\alpha_1, \ldots, \alpha_n, k_1, \ldots, k_{n-1})=\\
((\id \alpha_1, {k_1}^{-1}\alpha_2, \ldots,  (k_{n-1}\cdots k_1)^{-1}\alpha_n),
( {k_1}^{-1}k_1, \ldots,  (k_{n-1}\cdots k_1)^{-1}k_{n-1}(k_{n-2}\cdots k_1)))
\\
(( \alpha_1, {k_1}^{-1}\alpha_2, \ldots,  (k_{n-1}\cdots k_1)^{-1}\alpha_n),
( \id, \ldots, \id)).
\end{multline*}
This generalized path is equal to the concatenation of the $n$ pieces $$\alpha_1(r)*{k_1}^{-1} \alpha_2(r)*(k_2k_1)^{-1} \alpha_3(r)*\cdots *  (k_{n-1}\cdots k_1)^{-1}\alpha_n(r)$$ since the connecting arrows are all identities. Moreover, $T$ satisfies the naturality condition and is continuous by the universal property of the colimit.

Therefore $\chi$ is an equivalence of categories between the groupoid of generalized paths and the groupoid of multiple $G$-paths. We will see next that the groupoid of multiple $G$-paths is just the free path space $X^I$ together with $G$ acting on it.

\subsubsection{The isomorphism of categories $\xi: P'=G\ltimes G\mmm(G\times I, X)\to P''=G\ltimes X^I$.} 
To construct this isomorphism we will use the fact that a multiple $G$-path $\sigma$  is determined by its value at the branch $\sigma_e$ corresponding to the interval $e\times I$,  since $\sigma$ is equivariant.

We define $\xi(\sigma)=\sigma i_e\in X^I$ on objects and 
$\xi(g,\sigma)=(g, \sigma i_e)$ on arrows. Conversely, $\xi^{-1}(\beta)=\sigma_{\beta}$ where
$\sigma_{\beta}(g,r)=g^{-1}\beta(r)$.
The functor $\xi: G\ltimes G\mmm(G\times I, X)\to G\ltimes X^I$ is an isomorphism of categories since it has a strict inverse functor: $\xi\circ\xi^{-1}=\id_{G\ltimes X^I}$ and $\xi^{-1}\circ\xi=\id_{G\ltimes GMap(G\times I, X)}$ satisfying that the restriction $\xi$ and the action $\xi^{-1}$ are both continuous.

\begin{thm}
All models for the path groupoid of $G\ltimes X$ are equivalent.
$$P(G\ltimes X)=\colim \varphi \ltimes \colim \psi\sim G\ltimes G\mmm(G\times I, X)= G\ltimes X^I.$$
\end{thm}

\begin{remark} We can also prove that any generalized map is equivalent to a strict map in the context of translation groupoids, without using groupoids atlases. It was proven by Pronk and Scull in \cite{Pronk:2010} that any generalized map
 $$I\overset{\delta}{\gets}\Int\overset{\beta}{\to}G\ltimes X$$ between translation groupoids is equivalent to a generalized map
 $$I\overset{\e}{\gets}G\ltimes Y\overset{\beta'}{\to}G\ltimes X$$
where the middle groupoid is a translation groupoid. In the same paper, they proved that any essential equivalence between translation groupoids has to be of some prescribed form. In our case, this implies that the essential equivalence
$I\overset{\e}{\gets}G\ltimes Y$ satisfies $e=G/K$ and $I=Y/K$ where $K$ acts freely on $Y$. Hence $K=G$ and $G$ acts freely on Y. Since $G$ acts also properly on $Y$ we have that $Y=G\times I$.  Then  any generalized map $(\delta, \beta)$ is equivalent to a generalized map 
$$I\overset{}{\gets}G\ltimes (G\times I)\overset{\beta''}{\to}G\ltimes X.$$
Now applying our isomorphism $\xi : P' \to P''$ to the right leg of the span, we obtain $\alpha= \xi(\beta'')\in X^I$ which gives the equivalent strict map
$I\overset{\alpha}{\to}G\ltimes X$.
\end{remark}

\subsection{Functoriality and Morita invariance of the path groupoid} 
In this section we will see that the path groupoid is functorial and that the path groupoid is well defined up to Morita equivalence.

\subsubsection{Functoriality}
We will show that an equivariant map between translation groupoids, induces an equivariant map between the path groupoids.

For a strict equivariant map $\varphi\ltimes f :G\ltimes X\to H\ltimes Y $ we have an induced map $\varphi_*\ltimes f_* :G\ltimes X^I\to H\ltimes Y^I $ defined by $f_*(\alpha)=f\circ \alpha$ for all $\alpha \in X^I$ and $\varphi_*=\varphi$.
We construct now an equivariant map $P(\varphi\ltimes f) :   P( G \ltimes X) \rightarrow P( H \ltimes Y)$ between the colimit constructions.

For every $n$ we have induced maps 
$$
\left (\varphi\ltimes f \right)_* : \mmm( I_{S_n},G\ltimes X) \rightarrow \mmm( I_{S_n},H\ltimes Y)
$$
in terms of the description $\mmm( I_{S_n},G\ltimes X)= G^n \ltimes (X^I)^n \times_{X^{n-1}} G^{n-1}$, this map corresponds just to $ \varphi^n \ltimes \left (f^{n-1}_* \times \varphi^n  \right)$. By taking the colimit we obtain an equivariant map $P(\varphi\ltimes f) :   P( G \ltimes X) \rightarrow P( H \ltimes Y)$.

Similarly, we have an induced map $\varphi_* \ltimes f_* : G\ltimes G\mmm(G\times I, X) \rightarrow H \ltimes H\mmm(H\times I, Y)$ between the multiple $G$-path groupoids. We consider 
an equivariant map $(G\times I)\overset{\sigma}{\to} X$ and define  $f_*(\sigma): H \times I \rightarrow Y$ by $f_*(\sigma)(h,r)=h^{-1} f( \sigma(e,r))$.

In any of the three models the functoriality is easy to check and we have,
\begin{thm}
The path groupoid of $G\ltimes X$ is functorial for equivariant maps.
\end{thm}
Moreover, the equivalence of the three models for the path groupoid is natural.
\begin{thm}
For a strict equivariant map $\varphi\ltimes f :G\ltimes X\to H\ltimes Y $ the following diagram is commutative:
$$
\xymatrix{ P(G\ltimes X) \ar[r]^{P(\varphi\ltimes f)} \ar[d]^{\chi}  &
P(H\ltimes Y) \ar[d]^{\chi}  \\
G \ltimes G\mmm(G\times I, X) \ar[r]^{\varphi_*\ltimes f_*} \ar[d]^{\xi} &H \ltimes H\mmm(H\times I, Y) \ar[d]^{\xi} \\
G \ltimes X^I \ar[r]^{\varphi_* \ltimes f_*} & H \ltimes Y^I
}
$$
\end{thm}

\subsubsection{Morita Invariance}
We will start by proving that 
an essential equivalence $ G\times X \rightarrow H \times Y$ induces an essential equivalence between the path groupoids, $P( G \ltimes X) \rightarrow P( H \ltimes Y)$.

This will give that for a given Morita equivalence
 \[G\ltimes X\overset{\sigma}{\gets}G'\ltimes X'\overset{\e}{\to}H\ltimes Y\] where $\e$ and $\sigma$ are essential equivalences, we have
induced essential equivalences
\[P(G\ltimes X)\overset{P(\sigma)}{\gets}P(G'\ltimes X')\overset{P(\e)}{\to}P(H\ltimes Y).\]

\begin{prop}\label{ess} If $\varphi\ltimes f :G\ltimes X\to H\ltimes Y $  is an essential equivalence, then $P(\varphi\ltimes f)$ is an essential equivalence.
\end{prop}

\begin{proof}
\begin{enumerate}
\item $P(\varphi\ltimes f)$ is fully faithful. We will show that $P( G \ltimes X)_1$ is  the following pullback of topological spaces:
\[\xymatrix{
P( G \ltimes X)_1 \ar[r]^{\e} \ar[d]_{(s,t)}& P( H \ltimes Y)_1 \ar[d]^{(s,t)} \\
P( G \ltimes X)_0\times P( G \ltimes X)_0 \ar[r]^{\e \times \e} & P( H \ltimes Y)_0\times P( H \ltimes Y)_0}\]

Specifically we have to prove that the natural map $\xi$  from $P( G \ltimes X)_1$ to the fibered product $ P( G \ltimes X)_0\times P( G \ltimes X)_0 \times_{P( H \ltimes Y)_0\times P( H \ltimes Y)_0)} P( H \ltimes Y)_1$ is a homeomorphism.

Let us define the inverse map $\xi^{-1}$. For $\alpha,\beta \in P( G \ltimes X)_0$ and an element $\theta \in P( H \ltimes Y)_1$ with $s(\theta)=\alpha,t(\theta)=\alpha$ we can assume that there is a subdivision of the interval such that $\alpha,\beta$ are represented both by elements of $\mmm( I_{S_n},G\ltimes X)$ and $\theta$ by an element of $H^n$. Therefore we have $\alpha=(\alpha_1,  \dots, \alpha_n,k_1,\dots,k_{n-1})$ and $\beta=(\beta_1,  \dots, \beta_n,k'_1,  \dots, k'_{n-1})$  such that $f(\beta_i(r))=h_i f(\alpha_i(r))$ for all $i=1, \ldots, n$ and $k'_i=h_{i+1} k_i {h_i}^{-1}$ for all $i=1, \ldots, n-1$.

But then by fixing $r$ and using that $\varphi\ltimes f$ is an essential equivalence we have a fibered product of topological spaces

\[\xymatrix{
 (G \ltimes X)_1 \ar[r]^{\e} \ar[d]_{(s,t)}& ( H \ltimes Y)_1 \ar[d]^{(s,t)} \\
( G \ltimes X)_0\times ( G \ltimes X)_0 \ar[r]^{\e \times \e} & ( H \ltimes Y)_0\times ( H \ltimes Y)_0}\]
and therefore for every $r \in I_i$ there is $g^r_i\in G$ such that $\phi(g^r_i)=h_i$. Since $G$ is discrete and the dependence on $r$ is continuous, the $n$-tuple $(g^r_1,\ldots,g^r_n)$ actually does not depend on $r$ and represents an element of $P( G \ltimes X)_1$.

\item $P(\varphi\ltimes f)$ is essentially surjective. We will show that \[s\pi_{1}: P( H \ltimes Y)_1\times^t_{P( H \ltimes Y)_0} P( G \ltimes X)_0\rightarrow P( H \ltimes Y)_0\] is an open surjection.

For \'etale groupoids 
the condition that the morphism $s\pi_{1}:  (H \ltimes Y)_1\times^t_{ (H \ltimes Y)_0} ( G \ltimes X)_0\rightarrow ( H \ltimes Y)_0$ is an open surjection implies that it has local sections. We will use these local sections to construct local sections of $s\pi_{1}: P( H \ltimes Y)_1\times^t_{P( H \ltimes Y)_0} P( G \ltimes X)_0\rightarrow P( H \ltimes Y)_0$.

Let $\{U_\alpha\}_{\alpha\in \Delta}$ be a cover of $Y$ and $s_\alpha : U_\alpha \rightarrow (H \ltimes Y)_1\times^t_{ (H \ltimes Y)_0} ( G \ltimes X)_0$ the local sections. Take  $\gamma \in P( H \ltimes Y)_0$ and suppose that $\gamma$ is represented by an element of  $\mmm( I_{S_n},H\ltimes Y)$ therefore $\gamma = (\gamma_1,  \dots, \gamma_n,k_1,\dots,k_{n-1})$  with $\gamma_i : I_i \rightarrow Y$.

Given the subdivision $\{0=r_0\le r_1\le\cdots\le r_n=1\}$ associated to $\gamma$ and with an open cover of $[0,1]$ given by connected intervals  $\{I_i \mid 1\leq i \leq n\}$, we want to construct a refinement of the subdivision
$$
\{0=r_0\le s_0^{1}\leq\ldots \leq s_{m_1}^1=r_1 \leq 
\cdots \leq r_{i-1}=s_0^{i}\leq\ldots \leq s_{m_i}^i=r_i \leq \cdots \le r_n=1\}
$$
along with connected intervals $I^i_j$ with the property that  $\gamma_i(I_j^i) \subseteq U_{\alpha_j^i}$ for some $\alpha_j^i$. To construct the subdivision, first for $[r_{i-1},r_i]$ with non-empty interior, we consider the covering $\{ I_i \cap U_{\alpha}\}_{\alpha\in \Delta}$, by compactness of the interval $[r_{i-1},r_i]$ we can find a partition $r_{i-1}=s_0^{i} < \ldots < s_{m_i}^i=r_i$ such that each $\gamma_i ( [ s_{j-1}^{i},s_j^{i}] )$ is contained in some $U_{\alpha^i_j}$. Let $I^i_j$ be an open connected neighborhood of $[ s_{j-1}^{i},s_j^{i}]$ small enough such that $I^i_j \cap \{s^i_0,s^i_1,\ldots,s^i_{m_i}\}=\{s^i_{j-1},s^i_j\}$.

For the repeated elements $r_{i-1}=r_i$ is a matter of just shrinking the interval $I_i$ to get $\gamma_i(I_i) \subseteq U_{\alpha_i}$ for some  $\alpha_i$ and to obtain an object of the category of subdivisions.

With the local sections $s_{\alpha^i_j}$ we obtain maps $\pi_2s_{\alpha^i_j}(\gamma_i(r)) : I^i_j \rightarrow X$ and functions $\pi_1\pi_1s_{\alpha^i_j}(\gamma_i(r)) : I^i_j \rightarrow H$, since the intervals are connected and $H$ is a discrete group, actually these functions are constant and we have elements $h_j^i \in H$ with $s \pi_1 (f(\pi_2s_{\alpha^i_j}(\gamma_i(r))),h_j^i)= \gamma_i(r)$, i.e.
$$
 f(\pi_2s_{\alpha^i_j}(\gamma_i(r)))= h_j^i \gamma_i(r)
$$
for $r \in I^i_j$.

Note that $\left (h^i_j \right )^{-1} f(\pi_2s_{\alpha^i_j}(\gamma_i(s^i_j)))) = \left ( h^i_{j+1}\right )^{-1} f(\pi_2s_{\alpha^i_{j+1}}(\gamma_i(s^i_j)))$ (both are $\gamma_i(s^i_j)$) and therefore 
$$
h^i_{j+1} \left ( h^i_j \right )^{-1} f(\pi_2s_{\alpha^i_j}(\gamma_i(s^i_j))) =  f(\pi_2s_{\alpha^i_{j+1}}(\gamma_i(s^i_j))).
$$
Since $f$ is full and faithful, there is a $g^i_j \in G$ with $\phi(g^i_j)=h^i_{j+1} \left ( h^i_j \right )^{-1}$, such that 

$$
g^i_j \pi_2s_{\alpha^i_j}(\gamma_i(s^i_j)) =  \pi_2s_{\alpha^i_{j+1}}(\gamma_i(s^i_j))
$$
Similarly at the intersection points of two consecutive paths we have $k_i\gamma_i(r) = \gamma_{i+1}(r)$ for all $r\in \tilde I^i_j$ and therefore
$$
k_i \left (h_{m_i}^i \right )^{-1} f(\pi_2s_{\alpha^i_{m_i}}(\gamma_i(r_i))= k_i \gamma_i(r_i) = \gamma_{i+1}(r_i) = \left ( h_{0}^{i+1} \right )^{-1} f(\pi_2s_{\alpha^{i+1}_{0}}(\gamma_{i+1}(r_i))
$$
so
$$
h_{0}^{i+1} k_i \left ( h_{m_i}^i \right )^{-1} f(\pi_2s_{\alpha^i_{m_i}}(\gamma_i(r_i))=  f(\pi_2s_{\alpha^{i+1}_{0}}(\gamma_{i+1}(r_i))
$$
and since we have that $f$ is full and faithful
then, we have elements $g^i \in G$ with $\phi(g^i)=h_{0}^{i+1} k_i \left ( h_{m_i}^i \right )^{-1}$, such that
 $$
 g^i \pi_2s_{\alpha^i_{m_i}}(\gamma_i(r_i))=  \pi_2s_{\alpha^{i+1}_{0}}(\gamma_{i+1}(r_i)).
$$
Therefore we have a
generalized path $\left( \left (\pi_2s_{\alpha^i_j}(\gamma_i(r)) \right )_{i,j},g^1_1,g^1_2,\ldots,g^1_{m_1},g^1,g^2_1,\ldots,g^n_{m_n} \right )$ and elements $(h^1_1,h^1_2,\ldots,h^1_{m_1-1},\ldots,h^n_{m_n-1})$ of $H$ . By construction,
$$
(h^1_1,h^1_2,\ldots,h^1_{m_1-1},\ldots,h^n_{m_n-1}) \left (  \left (\pi_2s_{\alpha^i_j}(\gamma_i(r)) \right)_{i,j},g^1_1,g^1_2,\ldots,g^1_{m_1},g^1,g^2_1,\ldots,g^n_{m_n} \right )
$$
is 
$$
(\left. \gamma_1 \right |_{I^1_1}, \left. \gamma_1 \right |_{I^1_2} \dots, \left.\gamma_1 \right |_{I^1_{m_1}},\ldots,\left. \gamma_n \right |_{I^n_{m_n}},id,id,\ldots,k_1,id,\ldots,k_n).
$$
In the colimit this represents the same element as $(\gamma_1,  \dots, \gamma_n,k_1,\dots,k_{n-1})$. Therefore we have constructed local sections on the set

$$
\left\{
(\gamma_1,  \dots, \gamma_n,k_1,  \dots, k_{n-1})\in \mmm( I_{S_n},H\ltimes Y ) | \gamma_i ( [ s_{j-1}^{i},s_j^{i}] )\subseteq U_{\alpha^i_j} 
\right\}$$

which is an open set in the compact open topology of $\mmm( I_{S_n},H\ltimes Y)$.
\end{enumerate}
\end{proof}

Thus, we have proved that the path groupoid functor sends essential equivalences to essential equivalences and therefore the path groupoid is invariant under  Morita equivalence.
\begin{thm}
If $G \ltimes X\sim_M H \ltimes Y$, then $P(G \ltimes X)\sim_M P(H \ltimes Y)$.
\end{thm}

\subsection{The free loop groupoid $L(G\ltimes X)$}
We use the model of the path groupoid $P''=G\ltimes X^I$ to define the loop groupoid as the following pullback along the diagonal:

\begin{equation} 
\xymatrix{
&G \ltimes X^I  \ar[d]^{\ev}\\ \Delta : G \ltimes X\ar[r]& (G\times G) \ltimes (X\times X)}
\label{eq:loop}
\end{equation}

\begin{defn} The {\em free loop groupoid} $L(G\ltimes X)$ of a translation groupoid $G\ltimes X$ is 
$$L(G\ltimes X)=(G\times G)\ltimes L_0$$
where
$$L_0=
\{ (\beta, h,l)\in X^I\times G\times G |\beta(0)=hl^{-1}\beta(1)\}$$
and the group $G\times G$ acts on $L_0$ by $(a,b)(\beta, h,l)=(a\beta, bha^{-1},bla^{-1})$. 
\end{defn}

The following diagram depicts an arrow $(a,b)\in G\times G$ from $(\beta, h,l)$ to $(a\beta, bha^{-1},bla^{-1})$.
\tikzset{->-/.style={decoration={
  markings,
  mark=at position .5 with {\arrow{>}}},postaction={decorate}}}
\definecolor{qqqqff}{rgb}{0.,0.,1.}
\definecolor{ffqqqq}{rgb}{1.,0.,0.}
\begin{center}
\begin{tikzpicture}[line cap=round,line join=round,>=stealth,x=1.0cm,y=1.0cm]
\clip(8.,-3.) rectangle (15.,4.);
\draw [line width=1.6pt,color=ffqqqq] (8.98,2.96)-- (9.02,3.02)-- (9.1,3.12)-- (9.18,3.2)-- (9.24,3.24)-- (9.3,3.28)-- (9.38,3.32)-- (9.46,3.36)-- (9.54,3.4)-- (9.68,3.42)-- (9.8,3.44)-- (9.96,3.44)-- (10.1,3.46)-- (10.24,3.46)-- (10.4,3.44)-- (10.54,3.42)-- (10.64,3.4)-- (10.78,3.36)-- (10.88,3.32)-- (10.98,3.26)-- (11.1,3.18)-- (11.22,3.14)-- (11.4,3.02)-- (11.52,2.98)-- (11.62,2.92)-- (11.76,2.86)-- (11.88,2.8)-- (12.,2.76)-- (12.12,2.74)-- (12.22,2.7)-- (12.34,2.7)-- (12.44,2.68)-- (12.52,2.66)-- (12.62,2.64)-- (12.76,2.62)-- (12.86,2.62)-- (12.98,2.62)-- (13.1,2.62)-- (13.18,2.62)-- (13.3,2.62)-- (13.38,2.64)-- (13.46,2.66)-- (13.52,2.7)-- (13.6,2.74)-- (13.66,2.8)-- (13.74,2.84)-- (13.82,2.9)-- (13.92,2.98)-- (14.,3.06)-- (14.04,3.14)-- (14.1,3.2)-- (14.16,3.28)-- (14.22,3.3)-- (14.26,3.36)-- (14.32,3.42)-- (14.36,3.48);
\draw [line width=1.6pt,color=qqqqff] (9.04,-1.02)-- (9.08,-0.96)-- (9.16,-0.86)-- (9.24,-0.78)-- (9.3,-0.74)-- (9.36,-0.7)-- (9.44,-0.66)-- (9.52,-0.62)-- (9.6,-0.58)-- (9.74,-0.56)-- (9.86,-0.54)-- (10.02,-0.54)-- (10.16,-0.52)-- (10.3,-0.52)-- (10.46,-0.54)-- (10.6,-0.56)-- (10.7,-0.58)-- (10.84,-0.62)-- (10.94,-0.66)-- (11.04,-0.72)-- (11.16,-0.8)-- (11.28,-0.84)-- (11.46,-0.96)-- (11.58,-1.)-- (11.68,-1.06)-- (11.82,-1.12)-- (11.94,-1.18)-- (12.06,-1.22)-- (12.18,-1.24)-- (12.28,-1.28)-- (12.4,-1.28)-- (12.5,-1.3)-- (12.58,-1.32)-- (12.68,-1.34)-- (12.82,-1.36)-- (12.92,-1.36)-- (13.04,-1.36)-- (13.16,-1.36)-- (13.24,-1.36)-- (13.36,-1.36)-- (13.44,-1.34)-- (13.52,-1.32)-- (13.58,-1.28)-- (13.66,-1.24)-- (13.72,-1.18)-- (13.8,-1.14)-- (13.88,-1.08)-- (13.98,-1.)-- (14.06,-0.92)-- (14.1,-0.84)-- (14.16,-0.78)-- (14.22,-0.7)-- (14.28,-0.68)-- (14.32,-0.62)-- (14.38,-0.56)-- (14.42,-0.5);
\draw [->-] (9.010769230769226,3.0061538461538437) -- (9.,-1.);
\draw [->-] (14.356923076923085,3.4753846153846353) -- (14.398461538461536,-0.5323076923076937);
\draw [->-] (12.96,1.16) -- (14.356923076923085,3.4753846153846357);
\draw [->-] (12.96,1.16) -- (10.38,1.2);
\draw [->-] (12.96,1.16) -- (9.010769230769226,3.0061538461538437);
\draw [->-] (10.38,1.2) -- (9.,-1.);
\draw [->-] (10.38,1.2) -- (14.398461538461536,-0.5323076923076935);
\draw (11.76,3.7) node[anchor=north west] {$\beta$};
\draw (11.6,-1.26) node[anchor=north west] {$a\beta$};
\draw (14.54,1.52) node[anchor=north west] {$a$};
\draw (8.4,1.5) node[anchor=north west] {$a$};
\draw (10.62,2.7) node[anchor=north west] {$h$};
\draw (13.6,2.34) node[anchor=north west] {$l$};
\draw (13.06,0.8) node[anchor=north west] {$alb^{-1}$};
\draw (9.9,0.7) node[anchor=north west] {$ahb^{-1}$};
\draw (11.06,1.8) node[anchor=north west] {$b$};
\begin{scriptsize}
\draw [fill=black] (9.010769230769226,3.0061538461538437) circle (1.5pt);
\draw [fill=black] (9.,-1.) circle (1.5pt);
\draw [fill=black] (14.356923076923085,3.4753846153846353) circle (1.5pt);
\draw [fill=black] (14.398461538461536,-0.5323076923076935) circle (1.5pt);
\draw [fill=black] (12.96,1.16) circle (1.5pt);
\draw [fill=black] (10.38,1.2) circle (1.5pt);
\end{scriptsize}
\end{tikzpicture}
\end{center}

We will show that this groupoid $(G\times G) \ltimes L_0$ is Morita equivalent to the translation groupoid $G\ltimes L$ where 

$$L=\{ (\alpha, g)\in X^I\times G |\alpha(0)=g\alpha(1)\}$$ and the action is given by $(\alpha, g)\sim (k\alpha, kgk^{-1})$. The following diagram depicts an arrow $(k, (\alpha, g))$ between $(\alpha, g)$ and $(k\alpha, kgk^{-1})$.

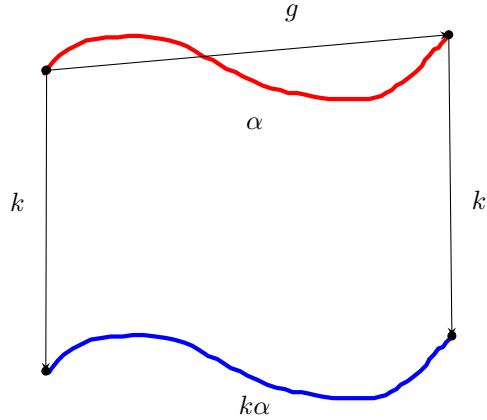
\begin{center}  
\begin{tikzpicture}[line cap=round,line join=round,>=stealth,x=1.0cm,y=1.0cm]
\clip(7.,-2.5) rectangle (16.,4.);
\draw [line width=1.6pt,color=ffqqqq] (8.98,2.96)-- (9.02,3.02)-- (9.1,3.12)-- (9.18,3.2)-- (9.24,3.24)-- (9.3,3.28)-- (9.38,3.32)-- (9.46,3.36)-- (9.54,3.4)-- (9.68,3.42)-- (9.8,3.44)-- (9.96,3.44)-- (10.1,3.46)-- (10.24,3.46)-- (10.4,3.44)-- (10.54,3.42)-- (10.64,3.4)-- (10.78,3.36)-- (10.88,3.32)-- (10.98,3.26)-- (11.1,3.18)-- (11.22,3.14)-- (11.4,3.02)-- (11.52,2.98)-- (11.62,2.92)-- (11.76,2.86)-- (11.88,2.8)-- (12.,2.76)-- (12.12,2.74)-- (12.22,2.7)-- (12.34,2.7)-- (12.44,2.68)-- (12.52,2.66)-- (12.62,2.64)-- (12.76,2.62)-- (12.86,2.62)-- (12.98,2.62)-- (13.1,2.62)-- (13.18,2.62)-- (13.3,2.62)-- (13.38,2.64)-- (13.46,2.66)-- (13.52,2.7)-- (13.6,2.74)-- (13.66,2.8)-- (13.74,2.84)-- (13.82,2.9)-- (13.92,2.98)-- (14.,3.06)-- (14.04,3.14)-- (14.1,3.2)-- (14.16,3.28)-- (14.22,3.3)-- (14.26,3.36)-- (14.32,3.42)-- (14.36,3.48);
\draw [line width=1.6pt,color=qqqqff] (9.04,-1.02)-- (9.08,-0.96)-- (9.16,-0.86)-- (9.24,-0.78)-- (9.3,-0.74)-- (9.36,-0.7)-- (9.44,-0.66)-- (9.52,-0.62)-- (9.6,-0.58)-- (9.74,-0.56)-- (9.86,-0.54)-- (10.02,-0.54)-- (10.16,-0.52)-- (10.3,-0.52)-- (10.46,-0.54)-- (10.6,-0.56)-- (10.7,-0.58)-- (10.84,-0.62)-- (10.94,-0.66)-- (11.04,-0.72)-- (11.16,-0.8)-- (11.28,-0.84)-- (11.46,-0.96)-- (11.58,-1.)-- (11.68,-1.06)-- (11.82,-1.12)-- (11.94,-1.18)-- (12.06,-1.22)-- (12.18,-1.24)-- (12.28,-1.28)-- (12.4,-1.28)-- (12.5,-1.3)-- (12.58,-1.32)-- (12.68,-1.34)-- (12.82,-1.36)-- (12.92,-1.36)-- (13.04,-1.36)-- (13.16,-1.36)-- (13.24,-1.36)-- (13.36,-1.36)-- (13.44,-1.34)-- (13.52,-1.32)-- (13.58,-1.28)-- (13.66,-1.24)-- (13.72,-1.18)-- (13.8,-1.14)-- (13.88,-1.08)-- (13.98,-1.)-- (14.06,-0.92)-- (14.1,-0.84)-- (14.16,-0.78)-- (14.22,-0.7)-- (14.28,-0.68)-- (14.32,-0.62)-- (14.38,-0.56)-- (14.42,-0.5);
 \draw[->-] (9.010769230769233,3.0061538461538544) -- (9.,-1.);
\draw [->-] (14.356923076923128,3.4753846153847006) -- (14.398461538461518,-0.5323076923077172);
\draw [->-] (9.,3.) -- (14.356923076923128,3.4753846153847006);
\draw (11.54,2.5) node[anchor=north west] {$\alpha$};
\draw (11.44,-1.2) node[anchor=north west] {$k\alpha$};
\draw (14.54,1.52) node[anchor=north west] {$k$};
\draw (8.4,1.5) node[anchor=north west] {$k$};
\draw (12.06,4) node[anchor=north west] {$g$};
\begin{scriptsize}
\draw [fill=black] (9.010769230769233,3.0061538461538544) circle (1.5pt);
\draw [fill=black] (9.,-1.) circle (1.5pt);
\draw [fill=black] (14.356923076923128,3.4753846153847006) circle (1.5pt);
\draw [fill=black] (14.398461538461518,-0.5323076923077174) circle (1.5pt);
\draw [fill=black] (9.,3.) circle (1.5pt);
\end{scriptsize}
\end{tikzpicture}
\end{center}  

\begin{prop} If $G\ltimes X$ is a topological groupoid, then the loop groupoid 
$L(G\ltimes X)$ is Morita equivalent to $G\ltimes L$, where $L$ and the action are defined above.
\end{prop}

\begin{proof}
We define an equivariant map $\phi\ltimes \e: (G\times G) \ltimes L_0 \to G\ltimes L$ by $\phi((a,b))=a$ and $\e((\beta, h, l))=(\beta, l^{-1}h)$. This map is an essential equivalence since  
the map
$$s\pi_1:  G\times L_0 \to L$$ given by $s\pi_1(k,(\beta, h,l))=(k^{-1}\beta, k^{-1}l^{-1}hk)$
is an open surjection and $G\times G\times L_0$ is given by the pullback of the following maps
\[\xymatrix{
&G\times L\ar[d]^{(s,t)} \\
L_0\times L_0 \ar[r]^{\e \times \e} & L\times L}\]
\end{proof}

\begin{remark}
We can use our description for the free loop groupoid in the special case of the point groupoid. We obtain that $L(G\ltimes \bullet)=(G\times G) \ltimes (G \times G)$ with the action $(a,b) \cdot (h,l)=(b h a^{-1},b l a^{-1})$. This groupoid is equivalent to $G$ acting on itself by conjugation by using the second characterization of the loop groupoid as $G\ltimes L$ with $L=\{ (\beta, g)\in X^I\times G |\beta(0)=g\beta(1)\}$. In this way, we recover a result of Lupercio and Uribe in \cite{Lupercio}. Observe that  $L(G\ltimes \bullet)=G\ltimes G$ whereas $P(G \ltimes \bullet )=G \ltimes \bullet$.

\end{remark}

\section{Based path and loop groupoids}
Now that we have defined the free path groupoid of a translation groupoid and have given several equivalent models, we can give an explicit characterization of the various groupoids resulting from fixing certain points. These  based groupoids of paths will be of great significance to the groupoid Lusternik-Schnirelmann theory defined in \cite{Colman2} as well as for the groupoid topological complexity defined in \cite{AC} and they will be further discussed elsewhere.

\subsection{The groupoid $\Omega_{x,y}$ of paths from $x$ to $y$}\label{based}
The groupoid of paths from $x$ to $y$,  $\Omega_{x,y}$, is defined as a pullback of the evaluation map  $\ev: P(G \ltimes X ) \to (G\times G) \ltimes (X\times X)$ and the constant map $x\times y: {\bf{1}}\to (G\times G) \ltimes (X\times X)$ where ${\bf{1}}$ is the  trivial groupoid with one object and one arrow,  ${\bf{1}}=e\ltimes \bullet$ and $(x\times y)(\bullet)=(x,y)$. That is,
$$\xymatrix{\Omega_{x,y}\ar[d]\ar[r]
&P(G \ltimes X ) \ar[d]^{\ev}\\ 
{\bf{1}}\ar[r]^{x\times y\;\;\;\;\;\;\;\;\;\;\;}& (G\times G) \ltimes (X\times X)}.
$$
Note that by the definition of groupoid pullback,  we have that if we take the model of the path groupoid of generalized paths, $P=\colim \phi \ltimes \colim \psi$, then the object space of the pullback is:
$$\{((\alpha_1, \cdots, \alpha_n, k_1, \cdots, k_{n-1}), h, l )\in \colim \psi \times(G\times G) |\alpha_1(0)=hx \mbox{ and } \alpha_n(1)=ly\}$$
i.e. the objects of $\Omega_{x,y}$ are sequences of paths and arrows $(h, \alpha_1, k_1,  \ldots, k_n, \alpha_n,  l )$ where
$s(k_i)=\alpha_{i+1}(r_i)$ for $i=1, \ldots, n-1$, 
$t(k_i)=\alpha_{i}(r_i)$ for $i=0, \ldots, n$, 
$s(h)=x$ and $s(l)=y$;
which are precisely the Haefliger $G$-paths \cite{Hae} when restricted to the closed intervals in the subdivision. Note that the sequences in Haefliger paths start and end with {\em arrows} and not with {\em paths} like our generalized paths in the free path groupoid defined in section \ref{path}. We recover the original sequence in the Haefliger $G$-paths when we fix the endpoints $x$ and $y$ in our free generalized paths.

For an equivalent characterization of the groupoid of paths from $x$ to $y$, we can consider our simplest model for the path groupoid $P''=G\ltimes X^I$. In this case, we describe the space of objects as
 $(\Omega_{x,y})_0=\{ (\beta, h,l)\in X^I\times (G\times G) |\beta(0)=hx \mbox{ and }\beta(1)=ly\}$.

These are paths that start at any point in the orbit of $x$ and end at any point in the orbit of $y$. The space of arrows is the cartesian product $G \times (\Omega_{x,y})_0$ where the action is given by $g(\beta, h, l)=(g\beta, gh, gl)$. 
 
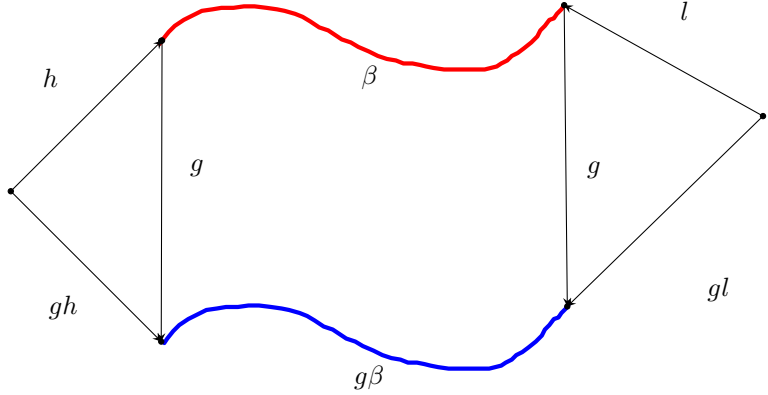
\begin{center} 
\begin{tikzpicture}[line cap=round,line join=round,>=stealth,x=1.0cm,y=1.0cm]
\clip(6.,-3.) rectangle (18.,5.);
\draw [line width=1.6pt,color=ffqqqq] (8.98,2.96)-- (9.02,3.02)-- (9.1,3.12)-- (9.18,3.2)-- (9.24,3.24)-- (9.3,3.28)-- (9.38,3.32)-- (9.46,3.36)-- (9.54,3.4)-- (9.68,3.42)-- (9.8,3.44)-- (9.96,3.44)-- (10.1,3.46)-- (10.24,3.46)-- (10.4,3.44)-- (10.54,3.42)-- (10.64,3.4)-- (10.78,3.36)-- (10.88,3.32)-- (10.98,3.26)-- (11.1,3.18)-- (11.22,3.14)-- (11.4,3.02)-- (11.52,2.98)-- (11.62,2.92)-- (11.76,2.86)-- (11.88,2.8)-- (12.,2.76)-- (12.12,2.74)-- (12.22,2.7)-- (12.34,2.7)-- (12.44,2.68)-- (12.52,2.66)-- (12.62,2.64)-- (12.76,2.62)-- (12.86,2.62)-- (12.98,2.62)-- (13.1,2.62)-- (13.18,2.62)-- (13.3,2.62)-- (13.38,2.64)-- (13.46,2.66)-- (13.52,2.7)-- (13.6,2.74)-- (13.66,2.8)-- (13.74,2.84)-- (13.82,2.9)-- (13.92,2.98)-- (14.,3.06)-- (14.04,3.14)-- (14.1,3.2)-- (14.16,3.28)-- (14.22,3.3)-- (14.26,3.36)-- (14.32,3.42)-- (14.36,3.48);
\draw [line width=1.6pt,color=qqqqff] (9.04,-1.02)-- (9.08,-0.96)-- (9.16,-0.86)-- (9.24,-0.78)-- (9.3,-0.74)-- (9.36,-0.7)-- (9.44,-0.66)-- (9.52,-0.62)-- (9.6,-0.58)-- (9.74,-0.56)-- (9.86,-0.54)-- (10.02,-0.54)-- (10.16,-0.52)-- (10.3,-0.52)-- (10.46,-0.54)-- (10.6,-0.56)-- (10.7,-0.58)-- (10.84,-0.62)-- (10.94,-0.66)-- (11.04,-0.72)-- (11.16,-0.8)-- (11.28,-0.84)-- (11.46,-0.96)-- (11.58,-1.)-- (11.68,-1.06)-- (11.82,-1.12)-- (11.94,-1.18)-- (12.06,-1.22)-- (12.18,-1.24)-- (12.28,-1.28)-- (12.4,-1.28)-- (12.5,-1.3)-- (12.58,-1.32)-- (12.68,-1.34)-- (12.82,-1.36)-- (12.92,-1.36)-- (13.04,-1.36)-- (13.16,-1.36)-- (13.24,-1.36)-- (13.36,-1.36)-- (13.44,-1.34)-- (13.52,-1.32)-- (13.58,-1.28)-- (13.66,-1.24)-- (13.72,-1.18)-- (13.8,-1.14)-- (13.88,-1.08)-- (13.98,-1.)-- (14.06,-0.92)-- (14.1,-0.84)-- (14.16,-0.78)-- (14.22,-0.7)-- (14.28,-0.68)-- (14.32,-0.62)-- (14.38,-0.56)-- (14.42,-0.5);
\draw [->-] (9.010769230769236,3.0061538461538597) -- (9.,-1.);
\draw [->-] (14.35692307692315,3.4753846153847334) -- (14.39846153846151,-0.5323076923077306);
\draw (11.54,2.8) node[anchor=north west] {$\beta$};
\draw (11.44,-1.2) node[anchor=north west] {$g\beta$};
\draw (14.54,1.52) node[anchor=north west] {$g$};
\draw (9.26,1.54) node[anchor=north west] {$g$};
\draw (15.78,3.64) node[anchor=north west] {$l$};
\draw [->-] (7.,1.) -- (9.010769230769236,3.0061538461538597);
\draw [->-] (7.,1.) -- (9.,-1.);
\draw [->-] (17.,2.) -- (14.35692307692315,3.4753846153847334);
\draw [->-] (17.,2.) -- (14.39846153846151,-0.5323076923077306);
\draw (16.14,-0.02) node[anchor=north west] {$gl$};
\draw (7.3,2.76) node[anchor=north west] {$h$};
\draw (7.38,-0.26) node[anchor=north west] {$gh$};
\begin{scriptsize}
\draw [fill=black] (9.010769230769236,3.0061538461538597) circle (1pt);
\draw [fill=black] (9.,-1.) circle (1pt);
\draw [fill=black] (14.35692307692315,3.4753846153847334) circle (1pt);
\draw [fill=black] (14.39846153846151,-0.5323076923077306) circle (1pt);
\draw [fill=black] (9.,3.) circle (1.pt);
\draw [fill=black] (7.,1.) circle (1pt);
\draw [fill=black] (17.,2.) circle (1pt);
\end{scriptsize}
\end{tikzpicture}
\end{center}

Since $(\beta, h,l)\sim (g\beta, gh, gl)$ for all $g\in G$, we can consider $g=h^{-1}$ and we have that all classes $[(\beta, h,l)]$ have a representative of the form $(\alpha, k)$ with $\alpha= h^{-1}\beta$ and $k=h^{-1}l$.
Then we can consider  the space of objects  $P_{x,y}=\{ (\alpha, k)\in X^I\times G |\alpha(0)=x \mbox{ and }\alpha(1)=ky\}$. Observe that 
$$(\alpha, k)\sim (h^{-1}\beta, e, h^{-1}l)\sim (gh^{-1}\beta,ge, gh^{-1}l)=  (g\alpha, g, gk)\sim (e\alpha, e, k)\sim (\alpha, k)$$
then the action is trivial on the space of objects $P_{x,y}$.

Therefore the {\em groupoid of paths between $x$ and $y$} is the translation groupoid $\Omega_{x,y}=G\ltimes (\Omega_{x,y})_0$ which is equivalent to the topological space $P_{x,y}$.

\subsection{The groupoid $\Omega_{x}$ of based loops}
Similarly, we define the based loop groupoid as the groupoid pullback,
$$\xymatrix{\Omega_{x}\ar[d]\ar[r]
&P(G \ltimes X ) \ar[d]^{\ev}\\ 
{\bf{1}} \ar[r]^{x\times x \;\;\; \;\;\; \;\;\; \;\;\; \;\;\;}& G\times G \ltimes (X\times X)}
$$
where $x\times x$ is the constant map with $(x\times x)(\bullet)=(x,x)$.

That is, the {\em based loop groupoid} is the translation groupoid $\Omega_{x}=G\ltimes (\Omega_{x})_0$ where the object space is  
$$(\Omega_{x})_0=\{ (\beta, h,l)\in X^I\times (G\times G) |\beta(0)=hx \mbox{ and }\beta(1)=lx\}$$ 
i.e. the space of paths that begin and end at (possibly different) points in the orbit of $x$. 
The action is given by $g (\beta, h,l)= (g\beta,gh,gl)$.
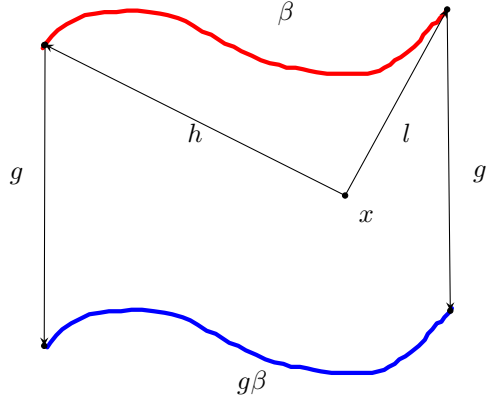
\begin{center} 
\begin{tikzpicture}[line cap=round,line join=round,>=stealth,x=1.0cm,y=1.0cm]
\clip(8.,-3.) rectangle (16.,4.);
\draw [line width=1.6pt,color=ffqqqq] (8.98,2.96)-- (9.02,3.02)-- (9.1,3.12)-- (9.18,3.2)-- (9.24,3.24)-- (9.3,3.28)-- (9.38,3.32)-- (9.46,3.36)-- (9.54,3.4)-- (9.68,3.42)-- (9.8,3.44)-- (9.96,3.44)-- (10.1,3.46)-- (10.24,3.46)-- (10.4,3.44)-- (10.54,3.42)-- (10.64,3.4)-- (10.78,3.36)-- (10.88,3.32)-- (10.98,3.26)-- (11.1,3.18)-- (11.22,3.14)-- (11.4,3.02)-- (11.52,2.98)-- (11.62,2.92)-- (11.76,2.86)-- (11.88,2.8)-- (12.,2.76)-- (12.12,2.74)-- (12.22,2.7)-- (12.34,2.7)-- (12.44,2.68)-- (12.52,2.66)-- (12.62,2.64)-- (12.76,2.62)-- (12.86,2.62)-- (12.98,2.62)-- (13.1,2.62)-- (13.18,2.62)-- (13.3,2.62)-- (13.38,2.64)-- (13.46,2.66)-- (13.52,2.7)-- (13.6,2.74)-- (13.66,2.8)-- (13.74,2.84)-- (13.82,2.9)-- (13.92,2.98)-- (14.,3.06)-- (14.04,3.14)-- (14.1,3.2)-- (14.16,3.28)-- (14.22,3.3)-- (14.26,3.36)-- (14.32,3.42)-- (14.36,3.48);
\draw [line width=1.6pt,color=qqqqff] (9.04,-1.02)-- (9.08,-0.96)-- (9.16,-0.86)-- (9.24,-0.78)-- (9.3,-0.74)-- (9.36,-0.7)-- (9.44,-0.66)-- (9.52,-0.62)-- (9.6,-0.58)-- (9.74,-0.56)-- (9.86,-0.54)-- (10.02,-0.54)-- (10.16,-0.52)-- (10.3,-0.52)-- (10.46,-0.54)-- (10.6,-0.56)-- (10.7,-0.58)-- (10.84,-0.62)-- (10.94,-0.66)-- (11.04,-0.72)-- (11.16,-0.8)-- (11.28,-0.84)-- (11.46,-0.96)-- (11.58,-1.)-- (11.68,-1.06)-- (11.82,-1.12)-- (11.94,-1.18)-- (12.06,-1.22)-- (12.18,-1.24)-- (12.28,-1.28)-- (12.4,-1.28)-- (12.5,-1.3)-- (12.58,-1.32)-- (12.68,-1.34)-- (12.82,-1.36)-- (12.92,-1.36)-- (13.04,-1.36)-- (13.16,-1.36)-- (13.24,-1.36)-- (13.36,-1.36)-- (13.44,-1.34)-- (13.52,-1.32)-- (13.58,-1.28)-- (13.66,-1.24)-- (13.72,-1.18)-- (13.8,-1.14)-- (13.88,-1.08)-- (13.98,-1.)-- (14.06,-0.92)-- (14.1,-0.84)-- (14.16,-0.78)-- (14.22,-0.7)-- (14.28,-0.68)-- (14.32,-0.62)-- (14.38,-0.56)-- (14.42,-0.5);
\draw [->-] (9.010769230769238,3.0061538461538624) -- (9.,-1.);
\draw [->-] (14.356923076923161,3.47538461538475) -- (14.398461538461506,-0.5323076923077372);
\draw (11.98,3.7) node[anchor=north west] {$\beta$};
\draw (11.44,-1.2) node[anchor=north west] {$g\beta$};
\draw (14.58,1.52) node[anchor=north west] {$g$};
\draw (8.42,1.5) node[anchor=north west] {$g$};
\draw (13.64,2.08) node[anchor=north west] {$l$};
\draw [->-] (13.,1.) -- (9.010769230769238,3.0061538461538624);
\draw (10.78,2.) node[anchor=north west] {$h$};
\draw [->-] (13.,1.) -- (14.356923076923161,3.47538461538475);
\draw (13.06,0.92) node[anchor=north west] {$x$};
\begin{scriptsize}
\draw [fill=black] (9.010769230769238,3.0061538461538624) circle (1pt);
\draw [fill=black] (9.,-1.) circle (1pt);
\draw [fill=black] (14.356923076923161,3.47538461538475) circle (1pt);
\draw [fill=black] (14.398461538461506,-0.5323076923077373) circle (1pt);
\draw [fill=black] (9.,3.) circle (1pt);
\draw [fill=black] (13.,1.) circle (1pt);
\end{scriptsize}
\end{tikzpicture}
\end{center}  
Again, the groupoid $\Omega_{x}$ is equivalent to the topological space $P_{x,x}=\{(\alpha, k)| \alpha(0)=x \mbox{ and }\alpha(1)=kx\}$.

Alternatively, the based loop groupoid $\Omega_{x}$ can be obtained as the following groupoid pullback:
$$\xymatrix{\Omega_{x}\ar[d]\ar[r]&L(G \ltimes X ) \ar[d]^{\ev_0}\\ 
{\bf{1}} \ar[r]^{x\;\;\; \;\;\; \;\;\; }&  G \ltimes  X}
$$
where $L(G \ltimes X )$ is the free loop groupoid.

\subsection{The groupoid $P_{x}$ of paths from $x$}
We define the $x$-based path groupoid as the following groupoid pullback:
$$\xymatrix{P_{x} \ar[d]\ar[r]&P(G \ltimes X ) \ar[d]^{\ev}\\ 
{{\bf{1}}\times (G\ltimes X)} \ar[r]^{(x,\id) \;\;\;\;\;\;}& (G\times G) \ltimes (X\times X)}
$$
where  $(x,\id) : {\bf{1}}\times (G\ltimes X)\to  (G\times G) \ltimes (X\times X)$ is given by $(x,\id)(\bullet,z)=(x,z).$
Then the object space of the pullback $P_{x}$ is $(P_{x})_0=\{ (\beta, (h, l),(\bullet,z) )\in X^I\times G\times G \times {\bf{1}}\times  X|\;\beta(0)=hx {\mbox{ and }} \beta(1)=lz\}=
\{ (\beta, (h, l),z )|\;\beta(0)=hx {\mbox{ and }} \beta(1)=lz\}.$
The group $G\times G$ acts on $(P_{x})_0$ by $(g,k) (\beta, (h, l), z)=(g\beta, (gh, glk^{-1}), kz)$.
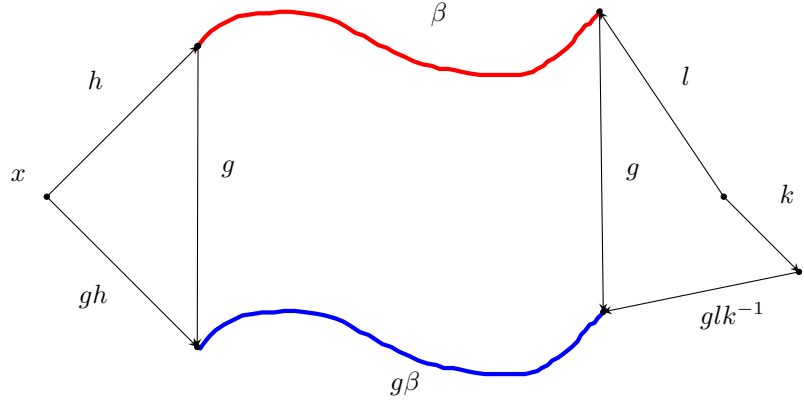
\begin{center} 
\begin{tikzpicture}[line cap=round,line join=round,>=stealth,x=1.0cm,y=1.0cm]
\clip(6.,-3.) rectangle (18.,4.);
\draw [line width=1.6pt,color=ffqqqq] (8.98,2.96)-- (9.02,3.02)-- (9.1,3.12)-- (9.18,3.2)-- (9.24,3.24)-- (9.3,3.28)-- (9.38,3.32)-- (9.46,3.36)-- (9.54,3.4)-- (9.68,3.42)-- (9.8,3.44)-- (9.96,3.44)-- (10.1,3.46)-- (10.24,3.46)-- (10.4,3.44)-- (10.54,3.42)-- (10.64,3.4)-- (10.78,3.36)-- (10.88,3.32)-- (10.98,3.26)-- (11.1,3.18)-- (11.22,3.14)-- (11.4,3.02)-- (11.52,2.98)-- (11.62,2.92)-- (11.76,2.86)-- (11.88,2.8)-- (12.,2.76)-- (12.12,2.74)-- (12.22,2.7)-- (12.34,2.7)-- (12.44,2.68)-- (12.52,2.66)-- (12.62,2.64)-- (12.76,2.62)-- (12.86,2.62)-- (12.98,2.62)-- (13.1,2.62)-- (13.18,2.62)-- (13.3,2.62)-- (13.38,2.64)-- (13.46,2.66)-- (13.52,2.7)-- (13.6,2.74)-- (13.66,2.8)-- (13.74,2.84)-- (13.82,2.9)-- (13.92,2.98)-- (14.,3.06)-- (14.04,3.14)-- (14.1,3.2)-- (14.16,3.28)-- (14.22,3.3)-- (14.26,3.36)-- (14.32,3.42)-- (14.36,3.48);
\draw [line width=1.6pt,color=qqqqff] (9.04,-1.02)-- (9.08,-0.96)-- (9.16,-0.86)-- (9.24,-0.78)-- (9.3,-0.74)-- (9.36,-0.7)-- (9.44,-0.66)-- (9.52,-0.62)-- (9.6,-0.58)-- (9.74,-0.56)-- (9.86,-0.54)-- (10.02,-0.54)-- (10.16,-0.52)-- (10.3,-0.52)-- (10.46,-0.54)-- (10.6,-0.56)-- (10.7,-0.58)-- (10.84,-0.62)-- (10.94,-0.66)-- (11.04,-0.72)-- (11.16,-0.8)-- (11.28,-0.84)-- (11.46,-0.96)-- (11.58,-1.)-- (11.68,-1.06)-- (11.82,-1.12)-- (11.94,-1.18)-- (12.06,-1.22)-- (12.18,-1.24)-- (12.28,-1.28)-- (12.4,-1.28)-- (12.5,-1.3)-- (12.58,-1.32)-- (12.68,-1.34)-- (12.82,-1.36)-- (12.92,-1.36)-- (13.04,-1.36)-- (13.16,-1.36)-- (13.24,-1.36)-- (13.36,-1.36)-- (13.44,-1.34)-- (13.52,-1.32)-- (13.58,-1.28)-- (13.66,-1.24)-- (13.72,-1.18)-- (13.8,-1.14)-- (13.88,-1.08)-- (13.98,-1.)-- (14.06,-0.92)-- (14.1,-0.84)-- (14.16,-0.78)-- (14.22,-0.7)-- (14.28,-0.68)-- (14.32,-0.62)-- (14.38,-0.56)-- (14.42,-0.5);
\draw [->-] (9.010769230769213,3.006153846153825) -- (9.,-1.);
\draw (11.98,3.7) node[anchor=north west] {$\beta$};
\draw (11.44,-1.2) node[anchor=north west] {$g\beta$};
\draw (14.58,1.56) node[anchor=north west] {$g$};
\draw (9.2,1.6) node[anchor=north west] {$g$};
\draw (15.32,2.86) node[anchor=north west] {$l$};
\draw [->-] (7.,1.) -- (9.010769230769213,3.006153846153825);
\draw (7.42,2.8) node[anchor=north west] {$h$};
\draw (6.4,1.48) node[anchor=north west] {$x$};
\draw [->-] (7.,1.) -- (9.,-1.);
\draw [->-] (16.,1.) -- (14.347692307692348,3.4615384615385323);
\draw [->-] (16.,1.) -- (17.,0.);
\draw [->-] (17.,0.) -- (14.398461538461488,-0.5323076923077571);
\draw [->-] (14.347692307692348,3.4615384615385323) -- (14.398461538461488,-0.5323076923077572);
\draw (16.62,1.28) node[anchor=north west] {$k$};
\draw (15.56,-0.28) node[anchor=north west] {$glk^{-1}$};
\draw (7.3,-0.04) node[anchor=north west] {$gh$};
\begin{scriptsize}
\draw [fill=black] (9.010769230769213,3.006153846153825) circle (1pt);
\draw [fill=black] (9.,-1.) circle (1pt);
\draw [fill=black] (14.398461538461488,-0.5323076923077571) circle (1pt);
\draw [fill=black] (9.,3.) circle (1pt);
\draw [fill=black] (7.,1.) circle (1pt);
\draw [fill=black] (16.,1.) circle (1pt);
\draw [fill=black] (14.347692307692348,3.4615384615385323) circle (1pt);
\draw [fill=black] (17.,0.) circle (1pt);
\end{scriptsize}
\end{tikzpicture}
\end{center}  
The {\em $x$-based path groupoid} is the translation groupoid $P_{x}=(G\times G)\ltimes (P_{x})_0$.

Observing that the equivalence class of each $(\beta, (h, l),z )\in (P_{x})_0$  contains an element of the form $(\alpha, g,w )\in X^I\times G \times X$ we have that the based path groupoid $P_{x}$ is equivalent to $G\ltimes P$ where 
$P=\{ (\alpha, g,w )|\;\alpha(0)=x {\mbox{ and }} \alpha(1)=gw\}$
and the action is given by $k (\alpha, g,w )= (\alpha, gk^{-1},kw )$. The following diagram depicts an arrow $(k ,(\alpha, g,w ))\in G\times P$ between $(\alpha, g,w )$ and $(\alpha, gk^{-1},kw )$.
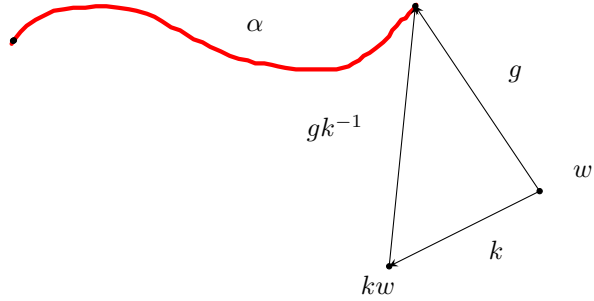
\begin{center} 
\begin{tikzpicture}[line cap=round,line join=round,>=stealth,x=1.0cm,y=1.0cm]
\clip(8.,-2.) rectangle (18.,4.);
\draw [line width=1.6pt,color=ffqqqq] (8.98,2.96)-- (9.02,3.02)-- (9.1,3.12)-- (9.18,3.2)-- (9.24,3.24)-- (9.3,3.28)-- (9.38,3.32)-- (9.46,3.36)-- (9.54,3.4)-- (9.68,3.42)-- (9.8,3.44)-- (9.96,3.44)-- (10.1,3.46)-- (10.24,3.46)-- (10.4,3.44)-- (10.54,3.42)-- (10.64,3.4)-- (10.78,3.36)-- (10.88,3.32)-- (10.98,3.26)-- (11.1,3.18)-- (11.22,3.14)-- (11.4,3.02)-- (11.52,2.98)-- (11.62,2.92)-- (11.76,2.86)-- (11.88,2.8)-- (12.,2.76)-- (12.12,2.74)-- (12.22,2.7)-- (12.34,2.7)-- (12.44,2.68)-- (12.52,2.66)-- (12.62,2.64)-- (12.76,2.62)-- (12.86,2.62)-- (12.98,2.62)-- (13.1,2.62)-- (13.18,2.62)-- (13.3,2.62)-- (13.38,2.64)-- (13.46,2.66)-- (13.52,2.7)-- (13.6,2.74)-- (13.66,2.8)-- (13.74,2.84)-- (13.82,2.9)-- (13.92,2.98)-- (14.,3.06)-- (14.04,3.14)-- (14.1,3.2)-- (14.16,3.28)-- (14.22,3.3)-- (14.26,3.36)-- (14.32,3.42)-- (14.36,3.48);
\draw (11.98,3.4) node[anchor=north west] {$\alpha$};
\draw (13.5,0.) node[anchor=north west] {$kw$};
\draw (12.78,2.14) node[anchor=north west] {$gk^{-1}$};
\draw [->-] (16.,1.) -- (14.347692307692348,3.4615384615385323);
\draw [->-] (16.,1.) -- (14.,0.);
\draw (16.32,1.46) node[anchor=north west] {$w$};
\draw (15.2,0.46) node[anchor=north west] {$k$};
\draw (15.46,2.8) node[anchor=north west] {$g$};
\draw [->-] (14.,0.) -- (14.347692307692348,3.4615384615385323);
\begin{scriptsize}
\draw [fill=black] (9.010769230769213,3.006153846153825) circle (1pt);
\draw [fill=black] (9.,3.) circle (1pt);
\draw [fill=black] (16.,1.) circle (1pt);
\draw [fill=black] (14.347692307692348,3.4615384615385323) circle (1pt);
\draw [fill=black] (14.,0.) circle (1pt);
\end{scriptsize}
\end{tikzpicture}
\end{center}  
The $x$-based path groupoid $P_x$ is not in general equivalent to a topological space.

Given points $x,y\in X$,  our various path groupoids are related as follows:

$$\xymatrix{
& \Omega_{x,y}\ar[rr]^-{}\ar[d]_{}&&{\bf{1_y}} \ar[d]^-{}\\ 
\Omega_x \ar[r] & P_{x}\ar[r]^-{}\ar[d]_{}&G\ltimes X^I\ar[d]^{\ev_0}\ar[r]^-{\ev_1}&G\ltimes X\\ 
&{\bf{1_x}} \ar[r]^-{}&G\ltimes X&& }$$
where ${\bf{1_x}}=e\ltimes x$ and ${\bf{1_y}}=e\ltimes y$ and all diagrams are commutative up to a natural transformation.


\subsection{Examples} We will illustrate in this section the concepts described in the previous sections by calculating various path groupoids in some particular cases.

\subsubsection{Topological spaces} The free path groupoid of the topological space $X$ is $P(e\ltimes X)= e\ltimes X^I=X^I$ and the free loop groupoid is $L(e\ltimes X)= e\ltimes L$ where $L=\{\alpha\in X^I | \: \alpha(0)=\alpha(1)\}$. In this way we recover the classical free path and loop spaces of a topological space. Likewise, the based path and loop groupoids also coincide with the classical ones for topological spaces.

\subsubsection{Groups} For a point groupoid $G\ltimes \bullet$ we have shown before that the path groupoid is  itself and the loop groupoid is $(G\times G) \ltimes (G \times G)$ with the action $(a,b) \cdot (h,l)=(b h a^{-1},b l a^{-1})$ which is equivalent to $G$ acting on itself by conjugation, that is, $L(G\ltimes \bullet)=G\ltimes G$ and $P(G \ltimes \bullet )=G \ltimes \bullet$. The based loop groupoid is the unit groupoid $G$, as a discrete topological space. The based path groupoid of paths emanating from $\bullet$ is $G\ltimes G$.

\subsubsection{Free actions} If $G$ acts freely on a topological space $X$, we observe that the groupoid $G\ltimes X$ and the topological space $X/G$ are Morita equivalent. Then, we have that $P(G\ltimes X)= P(e\ltimes X/G)=e\ltimes (X/G)^I=(X/G)^I$ and the free loop groupoid is $L(G\ltimes X)= L(X/G)$ where $L(X/G)$ is the free loop space of the topological space $X/G$. In the same way, we have that the based groupoids coincide with the ones of the topological space $X/G$.

\subsubsection{Orbifolds} We proved that for developable orbifolds $G\ltimes X$, the free path groupoid is $P(G\ltimes X)= G\ltimes X^I$ and the free  loop groupoid is $L(G\ltimes X)= G\ltimes L$ where $L=\{(\alpha,g) \in X^I \times G | \: \alpha(0)=g\alpha(1)\}$. Also, the groupoid of paths between $x$ and $y$ is the topological space $P_{x,y}=\{ (\alpha, k)\in X^I\times G |\alpha(0)=x \mbox{ and }\alpha(1)=ky\}$, the groupoid of based loops is the topological space $P_{x,x}=\{(\alpha, k)| \alpha(0)=x \mbox{ and }\alpha(1)=kx\}$
and the groupoid of based paths from $x$ is the translation groupoid $P_{x}=(G\times G)\ltimes (P_{x})_0$.


\section{Homotopy}
We will define in this section a notion of homotopy based on the explicit description of the path groupoid $P(G\ltimes X)$ given in the previous section. This will provide a concrete alternative to the more abstract presentation given by Noohi in \cite{Noohi1, Noohi2} for stacks.
\subsection{Natural transformations for translation groupoids}
The equivariant maps $\varphi\ltimes f: K\ltimes Z \to G\ltimes X$ and $\psi\ltimes g: K\ltimes Z \to G\ltimes X$ are equivalent by a natural transformation if there exists a $K$-map $\gamma : Z \to G$ such that $\gamma(z)f(z)=g(z)$ for all $z\in Z$ where both $Z$ and $G$ are $K$-spaces considering the following action of $K$ on $G$:
$$K\times G \to G \mbox{ where } (k,g) \mapsto \psi(k) g \varphi(k)^{-1}.$$

Therefore  $\varphi\ltimes f\sim  \psi\ltimes g$ if there exists $\gamma : Z \to G$ such that
\begin{enumerate}
\item  $\gamma(z)f(z)=g(z)$ for all $z\in Z$ and 
\item $\gamma(kz)=\psi(k)\gamma(z)\varphi(k)^{-1}$ for all $k\in K$.
\end{enumerate}

If $Z$ is connected, then $\gamma$ is a constant map since $G$ is discrete. Then $\varphi\ltimes f\sim  \psi\ltimes g$ if there exists $h\in G$ such that $h f(z)=g(z)$ for all $z\in Z$ and $h=\psi(k) h\varphi(k)^{-1}$ for all $k\in K$. Then $g=h f$ and $\psi = h^{-1} \varphi h$. In other words, $\psi(k)$ is conjugated to $\varphi(k)$ for all $k\in K$.

If $G$ is abelian, then $\varphi\ltimes f\sim  \psi\ltimes g$ if $g=hf$ for some $h\in G$ and  $\varphi=\psi$. 

If $X=Z=\bullet$, then $\varphi\ltimes \bullet\sim \psi\ltimes \bullet$ iff $\varphi$ and $\psi$ are conjugate, $\varphi=h^{-1}\psi h$. In particular, when the group acting is abelian we have that two maps between point groupoids are equivalent by a natural transformation only if they are equal.

We give now a characterization of $2$-isomorphism for strict maps. Namely if two strict maps are $2$-isomorphic then when composed with an essential equivalence they are equivalent by a natural transformation, and if two strict maps are equivalent by a natural transformation then they are $2$-isomorphic as generalized maps.

\begin{prop}\label{naturaltransformationimpliesiso}
If $ f$ and $g$ are equivalent by a natural transformation, then  $ f\Rightarrow  g$ as generalized equivariant maps.
\end{prop}

\begin{proof} Just consider the essential equivalences $\eta$ and $\nu$ as identity maps and the following diagram is commutative up to natural transformations since $ f\sim  g$
$$ 
\xymatrix{ &
{G\ltimes X}\ar[dr]^{f}="0" \ar[dl]_{\id }="2"&\\
{G\ltimes X}&{G\ltimes X} \ar[u]_{\id} \ar[d]^{\id}&{H\ltimes Y}\\
&{G\ltimes X}\ar[ul]^{\id}="3" \ar[ur]_{g}="1"&
\ar@{}"0";"1"|(.4){\,}="7"
\ar@{}"0";"1"|(.6){\,}="8"
\ar@{}"7" ;"8"_{\sim}
\ar@{}"2";"3"|(.4){\,}="5"
\ar@{}"2";"3"|(.6){\,}="6"
\ar@{}"5" ;"6"^{\sim}
}$$
\end{proof}

\begin{prop}\label{2iso}
If two strict maps $ f :G\ltimes X\to H\ltimes Y $ and $ g :G\ltimes X\to H\ltimes Y $ are $2$-isomorphic, then there exists an essential equivalence $\eta :\cL\to G\ltimes X $ such that $f\eta\sim g \eta$.
\end{prop}

\begin{proof} We have that there exist essential equivalences $\eta, \nu$ such that the following diagram
$$ 
\xymatrix{ &
{G\ltimes X}\ar[dr]^{f}="0" \ar[dl]_{\id }="2"&\\
{G\ltimes X}&{\cL} \ar[u]_{\eta} \ar[d]^{\nu}&{H\ltimes Y}\\
&{G\ltimes X}\ar[ul]^{\id}="3" \ar[ur]_{g}="1"&
\ar@{}"0";"1"|(.4){\,}="7"
\ar@{}"0";"1"|(.6){\,}="8"
\ar@{}"7" ;"8"_{\sim}
\ar@{}"2";"3"|(.4){\,}="5"
\ar@{}"2";"3"|(.6){\,}="6"
\ar@{}"5" ;"6"^{\sim}
}$$
commutes up to natural transformation. That is $\eta\sim \nu$ and $f \eta\sim g \nu$. Therefore, $f \eta\sim g \eta$.
\end{proof}

\begin{prop}
If  $(\e, f )\Rightarrow  (\sigma, g ) $, then there exist essential equivalences $\nu$ and $\eta$  such that $f \nu\Rightarrow g \eta$.

\end{prop}

\begin{proof} By definition of $2$-isomorphism, there are essential equivalences $\nu$ and $\eta$  such that $f \nu\sim g \eta$. The result follows from Proposition \ref{naturaltransformationimpliesiso}.

\end{proof}

\begin{prop}\label{2isoimplies2iso}
If  $f\Rightarrow g$, then   $(\e, f )\Rightarrow  (\sigma, g )$ for all essential equivalences $\e$, $\sigma$ with $\e\sim\sigma$.
\end{prop}

\subsection{Diagonal map}
We will consider the pullback of the unique morphism $G\ltimes X\overset{c}{\to} {\bf{1}}$ with itself, where ${\bf{1}}$ is the terminal object in  $\MTG$. This pullback defines the product and then by the universal property we obtain the definition of the diagonal map. Then, the path groupoid will be a factorization of that diagonal.

\begin{defn}\cite{GJ}
An object $T$ in a bicategory $\B$ is terminal if the category $\B[C,T]$ is equivalent to the terminal category for every object $C$ in $\B$. A terminal object is unique up to equivalence when it exists.

\end{defn}

The trivial groupoid ${\bf{1}}=e\ltimes \bullet$ is the terminal object in the bicategory of translation groupoids $\MTrG$ since the category $\MTrG[G\ltimes X, {\bf{1}}]$  is equivalent to the category ${\bf{1}}$. Indeed, the objects in the category $\MTrG[G\ltimes X, {\bf{1}}]$ are generalized maps and the arrows are classes of diagrams. We can see that all objects are related by an arrow, i.e. $\MTrG[G\ltimes X, {\bf{1}}]$ is the pair groupoid. Given two generalized maps
$ G\ltimes X\overset{\e'}{\gets}G'\ltimes X'\overset{c'}{\to} {\bf{1}}$ and  $ G\ltimes X\overset{\e''}{\gets}G''\ltimes X''\overset{c''}{\to} {\bf{1}}$ we can see that they are equivalent:
$$
\xymatrix{ &
{G'\ltimes X'}\ar[dr]^{}="0" \ar[dl]_{\e' }="2"&\\
{G\ltimes X}&{P} \ar[u]_{} \ar[d]^{}&{{\bf{1}}}\\
&{G''\ltimes X''}\ar[ul]^{\e''}="3" \ar[ur]_{}="1"&
\ar@{}"0";"1"|(.4){\,}="7"
\ar@{}"0";"1"|(.6){\,}="8"
\ar@{}"7" ;"8"_{\sim_{}}
\ar@{}"2";"3"|(.4){\,}="5"
\ar@{}"2";"3"|(.6){\,}="6"
\ar@{}"5" ;"6"^{\sim}
}
$$
by considering $P$ as the pullback of $\e'$ and $\e''$. In particular, the strict constant map $G\ltimes X\overset{c}{\to} {\bf{1}}$ is the (unique up to 2-iso) map to the terminal object.

Let's now consider the pullback of this constant map with itself which defines the product:
$$
\xymatrix{
G \times G \ltimes (X \times X )  \ar[d]_{} \ar[r]_{}&{G\ltimes X}\ar[d]^{c}\\
{G\ltimes X} \ar[r]_{c}&{\bf{1}}
}
$$
The  product $(G\times G)\ltimes (X\times X)$ of the object $G\ltimes X$ with itself is unique up to equivalence. 

By the universal property of the pullback, we have that there exists a map $\Delta$ that makes the two triangles commutative up to natural transformation:

\begin{tikzcd}
G \ltimes X
\arrow[bend left]{drr}{\id}
\arrow[bend right]{ddr}{\id}
\arrow[dashed]{dr}{\Delta} & & \\
& (G\times G) \ltimes (X\times X) \arrow{r}{p_1} \arrow{d}{p_2}
& G\ltimes X \arrow{d}{c}   \\
& G\ltimes X  \arrow{r}{c} &{\bf{1}}
\end{tikzcd}

The map $\Delta: G \ltimes X\to (G\times G) \ltimes (X\times X)$ is the {\em diagonal map}. Its explicit definition on objects is $\Delta(x)=(x,x)$ and on arrows, $\Delta(g,x)=(g,g,x,x)$. The diagonal map is defined up to $2$-isomorphism. 
\begin{remark}
Note that the diagonal defined in \cite{Adem} is $2$-isomorphic to this one.
\end{remark}

\begin{defn} The {\em evaluation map} $\ev: G\ltimes X^I\to (G\times G)\ltimes (X\times X)$ is given by $\ev(g,\alpha)= (g,g,\alpha(0), \alpha(1))$.
\end{defn}
We have that the diagonal map factors through the path groupoid as expected.
\begin{prop}
There is a factorization of the diagonal map $\Delta$
$$
\xymatrix
{G\ltimes X \ar[dr]_{k} \ar[rr]^{\Delta\;\;\;\;\;\;\;\;\;}&& (G\times G) \ltimes (X\times X)\\
&G\ltimes X^I\ar[ru]_{e}&
}
$$
where $k$ and $e$ are generalized maps.
\end{prop}

\begin{proof}
 Let $k$ be the functor $G\ltimes X\to G\ltimes X^I$  given by  $x\rightsquigarrow  \alpha_x$ on objects, and $(g, x)\rightsquigarrow  (g, \alpha_x)$ where  $\alpha_x:I\to X$ is a constant path at $x\in X$ and let $e$ be the evaluation map, $e=\ev$.
 Then, we have that the composition  $e\circ c$ is equivalent by a natural transformation to the diagonal $\Delta$.
\end{proof}

\subsection{Homotopic maps}
We will give now an explicit  characterization  of the homotopy between generalized maps.
\begin{defn}
Two generalized maps  $K\ltimes Y\overset{\sigma}{\gets}K'\ltimes Y'\overset{f}{\to}G\ltimes X$  and  $K\ltimes Y\overset{\tau}{\gets}K''\ltimes Y''\overset{g}{\to}G\ltimes X$ are {\em homotopic} if there is a generalized map 
$K\ltimes Y\overset{\e}{\gets}\tilde K\ltimes \tilde Y\overset{H}{\to}G\ltimes X$ such that the following diagram commutes up to $2$-isomorphism

$$
\xymatrix{ G\ltimes X&&
{G \ltimes X^I}\ar[rr]^{\ev_{1}}="0" \ar[ll]_{\ev_0 }="2"&&G\ltimes X\\
&{K'\ltimes Y'}\ar[ul]^{f}\ar[dr]_{\sigma}&{\tilde K\ltimes \tilde Y} \ar[u]_{H} \ar[d]^{\e}&{K''\ltimes Y''}\ar[ur]_{g}\ar[dl]^{\tau}&\\
&&{K\ltimes Y}&&
}
$$

\end{defn}

This means that the generalized map $(\sigma, f)$ is isomorphic  to the generalized map $(\e, \ev_0\circ H)$ and $(\tau, g)$ is isomorphic  to  $(\e, \ev_1\circ H)$.

That is $(\sigma, f)$ is homotopic to $(\tau, g)$  if there exists $(\e,  H)$ and two  commutative diagrams up to natural transformations:

$$
\xymatrix{ &
{\tilde K\ltimes \tilde Y}\ar[dr]^{\ev_{0}H}="0" \ar[dl]_{\e }="2"&\\
{\cK\ltimes Y}&{\cL_0} \ar[u]_{u_0} \ar[d]^{v_0}&{G\ltimes X}\\
&{K'\ltimes Y'}\ar[ul]^{\sigma}="3" \ar[ur]_{f}="1"&
\ar@{}"0";"1"|(.4){\,}="7"
\ar@{}"0";"1"|(.6){\,}="8"
\ar@{}"7" ;"8"_{\sim}
\ar@{}"2";"3"|(.4){\,}="5"
\ar@{}"2";"3"|(.6){\,}="6"
\ar@{}"5" ;"6"^{\sim}
}\;\;\;\;\;\;\;
\xymatrix{ &
{\tilde K\ltimes \tilde Y}\ar[dr]^{\ev_{1}H}="0" \ar[dl]_{\e }="2"&\\
{\cK\ltimes Y}&{\cL_1} \ar[u]_{u_1} \ar[d]^{v_1}&{G\ltimes X}\\
&{K''\ltimes Y''}\ar[ul]^{\tau}="3" \ar[ur]_{g}="1"&
\ar@{}"0";"1"|(.4){\,}="7"
\ar@{}"0";"1"|(.6){\,}="8"
\ar@{}"7" ;"8"_{\sim_{}}
\ar@{}"2";"3"|(.4){\,}="5"
\ar@{}"2";"3"|(.6){\,}="6"
\ar@{}"5" ;"6"^{\sim}
}
$$
where $\cL_i$ is a translation groupoid, and $u_i$ and $v_i$ are equivariant essential  equivalences for $i=0,1$.
We will denote this homotopy between equivariant generalized maps by $\simeq$.

\begin{remark}
$(\sigma, f)\simeq (\tau, g)$  if $\exists (\e,  H)$ and essential equivalences $u_0, u_1, v_0, v_1$ such that 
$$f v_0 \sim \ev_0 H u_0\mbox{ and } g v_1 \sim \ev_1 H u_1$$
with $\e u_0 \sim \sigma v_0\mbox{ and } \e u_1 \sim \tau v_1$.
\end{remark}

\begin{prop}\label{isoimplieshomotopy} If $(\sigma, f)\Rightarrow (\tau, g)$, then $(\sigma, f)\simeq (\tau, g)$.
\end{prop}

\begin{proof}
Consider $H=i_X\circ f$ where $i_X$ is the inclusion of $X$ in $X^I$  given by  $i_X(x)=\alpha_x$ with $\alpha_x$ being the constant map $\alpha_x(t)=x$ for all $t\in I$. Then the following diagram is commutative up to $2$-isomorphism:
$$
\xymatrix{ G\ltimes X&&
{G \ltimes X^I}\ar[rr]^{\ev_{1}}="0" \ar[ll]_{\ev_0 }="2"&&G\ltimes X\\
&{K'\ltimes Y'}\ar[ul]^{f}\ar[dr]_{\sigma}&{ K'\ltimes Y'} \ar[u]_{H} \ar[d]^{\sigma}&{K''\ltimes Y''}\ar[ur]_{g}\ar[dl]^{\tau}&\\
&&{K\ltimes Y}&&
}
$$
The first triangle is an equality and the second one is commutative since $(\sigma, f)\Rightarrow (\tau, g)$.
\end{proof}

\begin{remark} Let $ f$ and $g$ be strict maps. Following the characterization for isomorphism of strict maps given in Proposition  \ref{2iso} and the definition of groupoid homotopy, we have that 
$f\simeq g$ if there exists a generalized map $(\e, H)$ and
essential equivalences $\eta$ and $\nu$ such that $f \e \eta \sim \ev_0 H \eta$ and $g \e\nu \sim \ev_1 H \nu$.
\end{remark}

\begin{prop} Let $ f$ and $g$ be strict maps.
\begin{enumerate}
\item If  $ f$ and $g$ are $\psi$-equivariantly homotopic maps,  then  $ f\simeq  g$ as generalized equivariant maps.

\item If  $f$ and $g$ are equivalent by a natural transformation, then  $ f\simeq  g$ as generalized equivariant maps.
\end{enumerate}
\end{prop}
\begin{proof}
\begin{enumerate}
\item Let $H: Y\to X^I$ be the $\psi$-equivariant homotopy, i.e. $H_t(ky)=\psi(k)H_t(y)$. Then the following diagram is commutative:
$$
\xymatrix{ G\ltimes X&
{G \ltimes X^I}\ar[r]^{\ev_{1}}="0" \ar[l]_{\ev_0 }="2"&G\ltimes X\\
&{K\ltimes Y}\ar[ul]^{f}\ar[u]_{H}\ar[ur]_{g}&
}
$$

\item Follows from Proposition \ref{naturaltransformationimpliesiso} and Corollary \ref{isoimplieshomotopy}.
\end{enumerate}

\end{proof}
Therefore our definition of homotopy generalizes both the notion of natural transformation and the notion of equivariant homotopy.

\begin{prop} If $(\e, f)\simeq (\sigma, g)$ then there exist essential equivalences $a$ and $b$ such that $fa\simeq gb$ as strict maps.
\end{prop} 

\begin{proof}
Since we have a homotopy between generalized maps, we know that there exists $(\delta, H)$ and essential equivalences $u_0, v_0, u_1, v_1$ such that $fv_0\sim \ev_0Hu_0, gv_1\sim \ev_1Hu_1$ and $\delta u_0\sim \e v_0$, $\delta u_1\sim \sigma v_1$. Take $a=v_0(u_0)^{-1}\delta^{-1}$ and $b=v_1(u_1)^{-1}\delta^{-1}$. We obtain that $fa$ and $gb$ are homotopic.
\end{proof}

\begin{prop}
The path groupoid $G\ltimes X^I$ is homotopy equivalent to the groupoid $G\ltimes X$. The evaluation 
$e_1 : G\ltimes X^I \to G\ltimes X $ is a homotopy equivalence.
\end{prop}
\begin{proof}
Consider the map $H: G\ltimes X^I \to G\ltimes (X^I)^I$ such that $H(\alpha)=\lambda$ with $\lambda: I\to X^I, \lambda(t)=\alpha(r+t-rt)$. We have the following commutative  diagram:
$$
\xymatrix{ G\ltimes X^I&
{G \ltimes (X^I)^I}\ar[r]^{\ev_{1}}="0" \ar[l]_{\ev_0 }="2"&G\ltimes X^I\\
&{G \ltimes X^I}\ar[ul]^{\id}\ar[u]_{H}\ar[ur]_{i\circ e_1}&
}
$$
showing that  $i\circ e_1$ is homotopic to the identity map.
\end{proof}


\section{Fibrations}
We recall the definition of fibration for topological spaces given as a dualization of the notion of cofibration.
\begin{defn}\cite{May, Strom}
A map $p: E \to B$ is a fibration if for all spaces $U$ with $\ev_0 \circ K= p\circ k$ in the diagram  
$$
\begin{tikzcd}
U
\arrow[bend left]{drr}{K}
\arrow[bend right]{ddr}{k}
\arrow[dashed]{dr}{\tilde K} & & \\
& E^I \arrow{r}{p_{*}} \arrow{d}{\ev_0}
& B^I \arrow{d}{\ev_0}   \\
& E  \arrow{r}{p} & B
\end{tikzcd}
$$
there exists $\tilde K$ that makes the diagram commute.
\end{defn}

We want to introduce a notion of fibration for generalized maps. First, let's note that a strict equivariant map $\varphi\ltimes f :G\ltimes X\to H\ltimes Y $
induces a map $\varphi_*\ltimes f_* :G\ltimes X^I\to H\ltimes Y^I $ by $f_*(\alpha)=f\circ \alpha$ for all $\alpha \in X^I$ and $\varphi_*=\varphi$. We proved in section \ref{ess} that
if $ \e :G\ltimes X\to H\ltimes Y $ is an essential equivalence, then $\e_* :G\ltimes X^I\to H\ltimes Y^I $ is an essential equivalence as well.

Then we have that every generalized map $G\ltimes X\overset{\e}{\gets}G'\ltimes X'\overset{f}{\to}H\ltimes Y$ induces a generalized map $G\ltimes X^I\overset{\e_*}{\gets}G'\ltimes X'^I\overset{f_*}{\to}H\ltimes Y^I$ between the path groupoids.

\begin{defn}
A generalized map $G\ltimes X\overset{\e}{\gets}G'\ltimes X'\overset{f}{\to}H\ltimes Y$ is a {\em groupoid fibration} if for all translation groupoids $L\ltimes U$ with $\ev_0 \circ (\Omega, K)  \Rightarrow (\omega, k)\circ (\e, f)$ in the diagram  
$$
\begin{tikzcd}
L\ltimes U \\
& \tilde \cL 
\arrow[dashed]{dr}[description]{\tilde K} \arrow[dashed]{ul}[description]{\tilde \Omega} 
& \cL \arrow[bend left]{drr}{K}  \arrow[bend right]{ull}{\Omega}& & \\
&\ell \arrow[bend right]{dr}{k} \arrow[bend left]{uul}{\omega}&G\ltimes X^I  \arrow{d}{\ev_0}  & G'\ltimes X'^I \arrow{r}{f_{*}} \arrow{d}{\ev_0}\arrow{l}{\e_*}& H\ltimes Y^I  \arrow{d}{\ev_0}   \\
&&G\ltimes X& G'\ltimes X'  \arrow{l}{\e}\arrow{r}{f} & H\ltimes Y
\end{tikzcd}
$$
there exists $(\tilde \Omega, \tilde K)$ that makes the diagram commute up to $2$-isomorphism.
\end{defn}

Since a $2$-isomorphism between strict maps induces a $2$-isomorphism between the induced maps between their path groupoids, we have that being a fibration is a property invariant under $2$-isomorphism.

\begin{prop} Consider  $2$-isomorphic maps $ f :G\ltimes X\to H\ltimes Y $ and $ g :G\ltimes X\to H\ltimes Y $, $f\Rightarrow g$. Then $f$ is a fibration if and only if $g$ is a fibration.
\end{prop}

We will see that for  $(\e, f)$ to be a groupoid fibration it is necessary and sufficient that the right leg of the span is a groupoid fibration (considered as a generalized map with identity as a left leg).

\begin{prop}
 A generalized map $G\ltimes X\overset{\e}{\gets}G'\ltimes X'\overset{f}{\to}H\ltimes Y$ is a groupoid fibration if and only if $ f :G'\ltimes X'\to H\ltimes Y $ is a groupoid fibration.

\end{prop}

\begin{proof} If the generalized map $(\e, f)$ is a groupoid fibration, we have that there exists $(\tau', \tilde H')$ that makes the diagram commute up to $2$-isomorphism:
$$
\begin{tikzcd}
L\ltimes U & & \cL \arrow[]{ddrr}{H}\arrow[]{ll}{\sigma}\\
& \tilde \cL \arrow[]{dr}{\tilde H'} \arrow[]{ul}{\tau'} & P  \arrow[dashed]{dr}{\tilde H}\arrow[dashed]{ull}{\tau}&  \\
&&G\ltimes X^I  \arrow{d}{\ev_0}   &G'\ltimes {X'}^I\arrow{r}{f_{*}} \arrow{l}{\e_*} \arrow{d}{\ev_0}& H\ltimes Y^I  \arrow{d}{\ev_0}   \\
&\ell \arrow{r}{\e H_0} \arrow[bend right]{rr}{H_0}\arrow[]{uuul}{\sigma_0}&G\ltimes X & G'\ltimes {X'} \arrow{l}{\e}\arrow{r}{f}&    H\ltimes Y
\end{tikzcd}
$$
Let $P$ be the following pullback 
$$
\begin{tikzcd}
P\arrow{r}{\tilde H''} \arrow{d}{\e'_*} & G'\ltimes X'^I  \arrow{d}{\e_*}   \\
 \tilde \cL\arrow{r}{\tilde H'}&    G\ltimes X^I
\end{tikzcd}
$$
Take $\tau=\tau'\e'_*$ and $\tilde H= \tilde H''$. Then $ f :G'\ltimes X'\to H\ltimes Y$ is a groupoid fibration.

Conversely, if $f$ is a groupoid fibration then we have this commutative diagram
$$
\begin{tikzcd}
L\ltimes U & & \cL \arrow[]{ddrr}{H}\arrow[]{ll}{\sigma}\\
&   & P \arrow[dashed]{d}{\tilde H} \arrow[]{dr}{\tilde H'}\arrow{ull}{\tau}&  \\
&\ell\arrow{dr}{H_0}\arrow[]{uul}{\sigma_0}&G\ltimes X^I  \arrow{d}{\ev_0}   &G'\ltimes {X'}^I\arrow{r}{f_{*}} \arrow{l}{\e_*} \arrow{d}{\ev_0}& H\ltimes Y^I  \arrow{d}{\ev_0}   \\
&\ell'  \arrow[bend right]{rr}{H'_0}\arrow[bend left]{uuul}{\sigma'_0}&G\ltimes X & G'\ltimes {X'} \arrow{l}{\e}\arrow{r}{f}&    H\ltimes Y
\end{tikzcd}
$$
where $(\sigma'_0, H'_0)=\e\circ (\sigma_0, H_0)$. Now take
$\tilde H=\e_*\tilde H'$ and $\tau=\tau'$. Therefore, $(\e, f)$ is a fibration.
\end{proof}
Then, the test to decide if a generalized map is a groupoid fibration amounts to check the definition of groupoid fibration with a strict map. Moreover,
 we know that any generalized map
 $L\ltimes U\overset{}{\gets}{\ell}\overset{}{\to}G\ltimes X$
is equivalent to a generalized map of the form
$L\ltimes U\overset{}{\gets}L'\ltimes U'\overset{}{\to}G\ltimes X$
where $L'$  may be chosen as $L\times G$ and the group homomorphisms  are the appropriate projections onto $L$ and $G$ \cite{Pronk:2010}.

The groupoid fibration definition specializes to the following

\begin{defn} A strict map $f:G\ltimes X\to H\ltimes Y $ is a groupoid fibration if for all translation groupoids $L\ltimes U$ with $\ev_0 \circ \;(\Omega, K)  \Rightarrow  f \circ (\omega, k)$ in the diagram

$$
\begin{tikzcd}
L\ltimes U \\
& \tilde \cL 
\arrow[dashed]{dr}{\tilde K} \arrow[dashed]{ul}{\tilde \Omega} 
& L''\ltimes U'' \arrow[bend left]{dr}{K}  \arrow[bend right]{ull}{\Omega}&  \\
&L'\ltimes U' \arrow[bend right]{dr}{k} \arrow[bend left]{uul}{\omega}&G\ltimes X^I  \arrow{d}{\ev_0}  \arrow{r}{f_{*}} & H\ltimes Y^I  \arrow{d}{\ev_0}   \\
&&G\ltimes X \arrow{r}{f}&     H\ltimes Y
\end{tikzcd}
$$
there exists $(\tilde \Omega, \tilde K)$ that makes the diagram commute up to $2$-isomorphism.
\end{defn}

In other words, $f$ is a groupoid fibration if for all commutative diagrams

$$
\xymatrix{ &
{ (L'')\ltimes U''}\ar[dr]^{\ev_{0}\circ K}="0" \ar[dl]_{\Omega }="2"&\\
{L\ltimes U}&{} \ar[u]_{\eta} \ar[d]^{\nu}&{H\ltimes Y}\\
&{(L\times G)\ltimes U'}\ar[ul]^{\omega}="3" \ar[ur]_{f\circ k}="1"&
\ar@{}"0";"1"|(.4){\,}="7"
\ar@{}"0";"1"|(.6){\,}="8"
\ar@{}"7" ;"8"_{\sim}
\ar@{}"2";"3"|(.4){\,}="5"
\ar@{}"2";"3"|(.6){\,}="6"
\ar@{}"5" ;"6"^{\sim}
}$$

there exists $(\tilde\Omega, \tilde K)$ such that the following diagrams commute:

$$
\xymatrix{ &
{\tilde \cL}\ar[dr]^{\ev_{0}\tilde K}="0" \ar[dl]_{\tilde\Omega }="2"&\\
{L''}&{} \ar[u]_{\eta'} \ar[d]^{\nu'}&{G\ltimes X}\\
&{(L\times G)\ltimes U'}\ar[ul]^{\omega}="3" \ar[ur]_{ k}="1"&
\ar@{}"0";"1"|(.4){\,}="7"
\ar@{}"0";"1"|(.6){\,}="8"
\ar@{}"7" ;"8"_{\sim}
\ar@{}"2";"3"|(.4){\,}="5"
\ar@{}"2";"3"|(.6){\,}="6"
\ar@{}"5" ;"6"^{\sim}
}\;\;\;\;\;\;\;
\xymatrix{ &
{\tilde \cL}\ar[dr]^{f_*\circ \tilde K}="0" \ar[dl]_{\tilde \Omega }="2"&\\
{L\ltimes U}&{} \ar[u]_{\eta''} \ar[d]^{\nu''}&{H\ltimes Y^I}\\
&{(L\times H)\ltimes U''}\ar[ul]^{\Omega}="3" \ar[ur]_{K}="1"&
\ar@{}"0";"1"|(.4){\,}="7"
\ar@{}"0";"1"|(.6){\,}="8"
\ar@{}"7" ;"8"_{\sim_{}}
\ar@{}"2";"3"|(.4){\,}="5"
\ar@{}"2";"3"|(.6){\,}="6"
\ar@{}"5" ;"6"^{\sim}
}
$$

\begin{prop}
The evaluation map $ev_0 : G \ltimes X^I  \rightarrow G  \ltimes X$ is a groupoid fibration.
\end{prop}
\begin{proof}
For all translation groupoids $L\ltimes U$ making the following diagram commutative up to $2$-isomorphism, we will construct the required generalized map $(\tilde \Omega, \tilde K)$.
$$
\begin{tikzcd}
L\ltimes U \\
& \tilde \cL 
\arrow[dashed]{dr}{\tilde K} \arrow[dashed]{ul}{\tilde \Omega} 
& L''\ltimes U'' \arrow[bend left]{dr}{K}  \arrow[bend right]{ull}{\Omega}&  \\
&L'\ltimes U' \arrow[bend right]{dr}{k} \arrow[bend left]{uul}{\omega}&G\ltimes {(X^I)}^I \arrow{d}{}  \arrow{r}{} & G\ltimes X^I  \arrow{d}{\ev_0}   \\
&&G\ltimes X^I \arrow{r}{ev_0}&     G\ltimes X
\end{tikzcd}
$$
Since there is a $2$-isomorphism between the generalized maps $(\Omega,ev_0 K)$ and $(\omega,ev_0 k)$, we know that there exists a groupoid $\tilde \cL$ and essential equivalences $\nu$ and $\eta$ such that the following diagram commutes up to natural transformations

$$
\xymatrix{ &
{ L'' \ltimes U''}\ar[dr]^{\ev_{0}\circ K}="0" \ar[dl]_{\Omega }="2"&\\
{L\ltimes U}&{} \ar[u]_{\eta} \tilde \cL \ar[d]^{\nu}&{G\ltimes X}\\
&{L'\ltimes U'}\ar[ul]^{\omega}="3" \ar[ur]_{ev_0 \circ k}="1"&
\ar@{}"0";"1"|(.4){\,}="7"
\ar@{}"0";"1"|(.6){\,}="8"
\ar@{}"7" ;"8"_{\sim}
\ar@{}"2";"3"|(.4){\,}="5"
\ar@{}"2";"3"|(.6){\,}="6"
\ar@{}"5" ;"6"^{\sim}
}$$
we take $\tilde\Omega= \Omega\eta$ and will construct a map $\tilde K :\tilde \cL \rightarrow G\ltimes {(X^I)}^I  $ such that the following diagram commutes up to natural transformations
$$
\begin{tikzcd}
L\ltimes U \\
& \tilde \cL  \arrow{ul}{\tilde\Omega} \arrow{r}{\eta} \arrow{d}{\nu} \arrow[dashed]{dr}{\tilde K} 
& L''\ltimes U'' \arrow[bend left]{dr}{K}  \arrow[bend right]{ull}{\Omega}&  \\
&L'\ltimes U' \arrow[bend right]{dr}{k} \arrow[bend left]{uul}{\omega}& G \ltimes {(X^I)}^I \arrow{r}{}\arrow{d}{} & G\ltimes X^I  \arrow{d}{\ev_0}   \\
&&G\ltimes X^I \arrow{r}{ev_0}&     G\ltimes X
\end{tikzcd}
$$

Consider the following groupoid pullback 
$$
\begin{tikzcd}
 P\arrow{r}{\pi_1} \arrow{d}{\pi_2}& G\ltimes X^I  \arrow{d}{ev_0}   \\
 G\ltimes X^I \arrow{r}{ev_0}&    G\ltimes X
\end{tikzcd}
$$

where $P$ is the translation groupoid $$P=(G \times G)  \ltimes (X^I \times_{X} X^I\times_{X} G)$$  with  $X^I \times_{X} X^I\times_{X} G= \{ (\alpha_1,\alpha_2,k) | k\alpha_1(0) =  \alpha_2 (0) \}$. We observe that in fact $P$ is equivalent to $\mmm(I_{S_2},G \ltimes X)$. To show this equivalence, we construct first a functor ${\overline K}: P\to G \ltimes (X^I \times_X X^I)$, where $X^I \times_X X^I = X^{I \vee I}$ is the pullback of the diagram
$$
\begin{tikzcd}
 &  X^I  \arrow{d}{ev_0}   \\
  X^I \arrow{r}{ev_0}&     X
\end{tikzcd}
$$
given by ${\overline K} (\alpha_1,\alpha_2,k) ) = (k\alpha_1,\alpha_2)$ on objects 
and  ${\overline K}(g_1,g_2)=g_2 $ on morphisms.

Since $(g_1,g_2) \cdot (\alpha_1,\alpha_2,k ) = (g_1 \alpha_1, g_2 \alpha_2,g_2 k g_1^{-1})$ and $${\overline K} (g_1 \alpha_1, g_2 \alpha_2,g_2 k g_1^{-1}) = (g_2 k g_1^{-1} g_1 \alpha_1,g_2 \alpha_2) = ( g_2 k \alpha_1, g_2 \alpha_2) = g_2 (k \alpha_1, \alpha_2) $$ we can see that this is just a special case of the equivalences of the path groupoid models from section \ref{path}.

$$
\mmm (I_{S_2}, G \ltimes X) \cong P \sim G \ltimes X^{I \vee I} \cong G \ltimes X^I
$$

We observe that the following diagram of functors commutes up to natural transformations

$$
\xymatrix{G \ltimes X^I&
{G \ltimes X^{I \vee I}}\ar[r]^{j_{2}^*}="0" \ar[l]_{j_1^* }="2"& G \ltimes X^I\\
&{P}\ar[ul]^{\pi_1}\ar[u]_{{\overline K}}\ar[ur]_{\pi_2}&
}
$$
since the RHS commutes on the nose and the LHS commutes up to a natural transformation. Here 
$j_1 :  I \rightarrow I \vee I$  and $j_2: I \rightarrow I \vee I$ are the natural maps for the coproduct of pointed spaces:
$$
\xymatrix{ I\ar[dr]_{j_1} \ar[r]^{i_1 }& I \times I\ar[d]^{\pi} & I \ar[dl]^{j_2} \ar[l]_{i_{2}} \\
&{I \vee I}&
}
$$
where $i_1(t)=(t,0), \; i_2(s)=(0,s)$ and $\pi :  I \times I \rightarrow I \vee I$ is a deformation retract.
Therefore, we have the following commutative diagram:
\begin{equation}\label{diagram1}
\begin{tikzcd}
P \arrow[dashed]{dr}{\overline{K}} \arrow[bend left]{ddrrr}{\pi_1} \arrow[bend right]{dddrr}{\pi_2} \\
& G \ltimes X^{I \vee I}\arrow[bend left]{drr}{j_1^*} \arrow[bend right]{ddr}{j_2^*}  
\ar[dashed]{dr}{\pi^*}  
  &  \\
& & G \ltimes X^{I \times I} \arrow{d}{i_2^*}  \arrow{r}{i_1^*} &  G \ltimes X^I  \arrow{d}{\ev_0}   \\
&& G \ltimes X^I \arrow{r}{ev_0}&  G \ltimes    X
\end{tikzcd}
\end{equation}

Now, by the universal property of the weak pullback, there exists a functor $\xi: \cL\to P$ such that the following diagram commutes up to natural transformation:

\begin{equation}\label{diagram2}
\begin{tikzcd}
\arrow[bend left]{drr}{K\eta}
\arrow[bend right]{ddr}{k\nu}
\cL
\arrow{rd}{\xi}&  & \\
&  P \arrow{r} \arrow{d}& G\ltimes X^I \arrow{d}{\ev_0}  \\
&  G\ltimes X^I \arrow{r}{ev_0} & X
\end{tikzcd}
\end{equation}

Now, we put together diagrams \ref{diagram1} and \ref{diagram2}

$$
\begin{tikzcd}
 \tilde \cL  \arrow{r}{\eta}   \arrow{d}{\nu} \arrow[dashed]{dr}{\xi}  & L''\ltimes U'' \arrow[bend left]{dddrrr}{K} &  &   & \\
L'\ltimes U' \arrow[bend right]{dddrrr}{k}  &P \arrow[dashed]{dr}{\overline{K}} \arrow[bend left]{ddrrr}{\pi_1} \arrow[bend right]{dddrr}{\pi_2} \\
&& G \ltimes X^{I \vee I} \arrow[drr,bend left=30,"{j_1^*}" pos=0.2]
 \arrow[bend right]{ddr}{j_2^*}  
\ar[dashed]{dr}{\pi^*}  
&  &  \\
&& & G \ltimes X^{I \times I} \arrow{d}{i_2^*}  \arrow{r}{i_1^*} &  G \ltimes X^I  \arrow{d}{\ev_0}   \\
&&& G \ltimes X^I \arrow{r}{ev_0}&  G \ltimes    X
\end{tikzcd}
$$

and define $\tilde K= \pi^* \circ \overline{K} \circ \xi$.

\end{proof}
\section*{Acknowledgments}
The first author acknowledges and thanks the financial support provided by the Max Planck Institute for Mathematics and 
the grant P12.160422.004/01 -FAPA ANDRES ANGEL from Vicedecanatura de Investigaciones de la Facultad de Ciencias de la Universidad de los Andes, Colombia.

The second author is partially supported by the Simons Foundation (\#278333). Collaboration between the co-authors was possible thanks to this financial support.

Part of this paper was written during the visit of the second author to the Universidad de los Andes and during the visit of the first author to Wright College. They gratefully acknowledge the host institutions for their hospitality. 

Both authors thank Gustavo Granja and  Behrang Noohi for their helpful insights on an earlier version of this manuscript. Special thanks to Dorette Pronk for valuable discussions.

\bibliography{AngelColman}

\begin{thebibliography}{10}

\bibitem{Adem}
A.~Adem, J.~Leida, and Y.~Ruan.
\newblock {\em Orbifolds and Stringy Topology}.
\newblock Cambridge University Press, 2007.
\newblock Cambridge Books Online.

\bibitem{Colman2}
H.~Colman.
\newblock The {L}usternik-{S}chnirelmann category of a lie groupoid.
\newblock {\em Transactions of the American Mathematical Society},
  362(10):5529--5567, 2010.

\bibitem{colman2011}
H.~Colman.
\newblock On the 1-homotopy type of {Lie} groupoids.
\newblock {\em Applied Categorical Structures}, 19(1):393--423, 2011.

\bibitem{AC}
H.~Colman and A.~Angel.
\newblock Groupoid {Topological} {Complexity}.
\newblock Preprint, 2015.

\bibitem{GJ}
N.~Gambino and A.~Joyal.
\newblock {\em On operads, bimodules and analytic functors}.
\newblock ArXiv e-prints, 2014.

\bibitem{Hae}
K.~Guruprasad and A.~Haefliger.
\newblock Closed geodesics on orbifolds.
\newblock {\em Topology}, 45(3):611-- 641, 2006.

\bibitem{Lupercio}
E.~Lupercio and B.~Uribe.
\newblock Loop groupoids, gerbes, and twisted sectors on orbifolds.
\newblock In {\em Orbifolds in mathematics and physics (Madison, WI, 2001),
  Contemp. Math 310}, 2001.

\bibitem{May}
P.~May.
\newblock {\em A concise course in algebraic topology}.
\newblock Chicago Lectures in Mathematics. University of Chicago Press,
  Chicago, IL, 1999.

\bibitem{Pronk1997}
I.~Moerdijk and D.~Pronk.
\newblock Orbifolds, sheaves and groupoids.
\newblock {\em $K$-Theory}, 12(1):3--21, 1997.

\bibitem{Noohi1}
B.~Noohi.
\newblock Mapping stacks of topological stacks.
\newblock {\em J REINE ANGEW MATH}, 2010(646):117--133, 2010.

\bibitem{Noohi2}
B.~Noohi.
\newblock Fibrations of topological stacks.
\newblock {\em Advances in Mathematics}, 252:612--640, 2014.

\bibitem{pardon}
J.~Pardon.
\newblock Enough vector bundles on orbispaces.
\newblock arXiv:1906.05816, 2020.

\bibitem{Pronk1996}
D.~Pronk.
\newblock Etendues and stacks as bicategories of fractions.
\newblock {\em Compositio Mathematica}, 102(3):243--303, 1996.

\bibitem{Pronk:2010}
D.~Pronk and L.~Scull.
\newblock Translation groupoids and orbifold cohomology.
\newblock {\em Canadian Journal of Mathematics}, 62(3):614--645, June 2010.

\bibitem{Pronk2017}
D.~Pronk, L.~Scull, and M.~Tommasini.
\newblock Atlases for ineffective orbifolds.
\newblock arXiv:1606.04439v2, 2017.

\bibitem{S}
I.~Satake.
\newblock On a generalization of the notion of manifold.
\newblock {\em Proc. Nat. Acad. Sci. U.S.A.}, 42:359--363, 1956.

\bibitem{thesis}
A.~Sibih.
\newblock Orbifold atlas groupoids.
\newblock Master thesis, 2013.

\bibitem{Strom}
J.~Strom.
\newblock {\em Modern Classical Homotopy Theory}.
\newblock Graduate studies in mathematics. American Mathematical Soc., 2011.

\end{thebibliography}
\bibliographystyle{plain}
\end{document}